\documentclass[11pt,a4paper]{amsart}
\pdfoutput=1
\usepackage{times}
\usepackage{amssymb}

\usepackage{ifthen}
\usepackage{cite}
\usepackage{graphicx}
\usepackage{xr}

\makeatletter
\newcommand{\ga}{\alpha}

\newcommand{\gk}{\kappa}
\newcommand{\gl}{\lambda}
\newcommand{\go}{\omega}
\newcommand{\gs}{\sigma}

\newcommand{\vt}{\vartheta}
\newcommand{\gD}{\Delta}

\newcommand{\gL}{\Lambda}
\newcommand{\gO}{\Omega}

\newcommand{\cA}{\mathcal{A}}
\newcommand{\cB}{\mathcal{B}}
\newcommand{\cC}{\mathcal{C}}
\newcommand{\cD}{\mathcal{D}}
\newcommand{\cE}{\mathcal{E}}
\newcommand{\cF}{\mathcal{F}}
\newcommand{\cG}{\mathcal{G}}
\newcommand{\cH}{\mathcal{H}}
\newcommand{\cI}{\mathcal{I}}

\newcommand{\cL}{\mathcal{L}}
\newcommand{\cM}{\mathcal{M}}

\newcommand{\cT}{\mathcal{T}}

\newcommand{\N}{\mathbb{N}}
\newcommand{\Q}{\mathbb{Q}}
\newcommand{\R}{\mathbb{R}}

\newcommand{\Z}{\mathbb{Z}}
\newcommand{\p}{\partial}

\newtheorem{theorem}{Theorem}[section]
\newtheorem{proposition}[theorem]{Proposition}
\newtheorem{lemma}[theorem]{Lemma}
\newtheorem{corollary}[theorem]{Corollary}

\theoremstyle{remark}

\newcommand{\erf}{\mathop{\operator@font erf}\nolimits}
\newcommand{\erfc}{\mathop{\operator@font erfc}\nolimits}
\newcommand{\sign}{\mathop{\operator@font sign}\nolimits}

\newenvironment{enum_a}
    {\begin{enumerate}}
    {\end{enumerate}}

\newif\if@golden  \@goldentrue
\newcommand{\f@ctor}{1}
\newlength{\aiv@width}  
\newlength{\aiv@height} 
\newlength{\tmp@width}  
\newlength{\tmp@height} 
\if@twocolumn\if@golden
  \else\fi
\else\if@golden\ifcase\@ptsize\relax\or
\or\fi
  \else\ifcase\@ptsize\relax\or
\or\fi\fi
\if@twocolumn\if@golden\ifcase\@ptsize\relax\renewcommand{\f@ctor}{53}
  \or\renewcommand{\f@ctor}{46}\or\renewcommand{\f@ctor}{43}\fi
  \else\ifcase\@ptsize\relax\renewcommand{\f@ctor}{51}\or
  \renewcommand{\f@ctor}{45}\or\renewcommand{\f@ctor}{42}\fi\fi
\else\if@golden\ifcase\@ptsize\relax \renewcommand{\f@ctor}{46}
  \or\renewcommand{\f@ctor}{43}\or\renewcommand{\f@ctor}{43}\fi
  \else\ifcase\@ptsize\relax\renewcommand{\f@ctor}{43}
  \or\renewcommand{\f@ctor}{40}\or\renewcommand{\f@ctor}{40}\fi\fi\fi
\multiply\textheight by \f@ctor
\let\comp\circ

\newcommand{\sgl}{\ensuremath\sqrt{2\gl}}

\newcommand{\cond}{\ensuremath\,\big|\,}

\newcommand{\Cii}{\ensuremath C_0^2(\cG)}

\newcommand{\eval}{\mathop{\big|}\nolimits}
\newcommand{\Eval}{\mathop{\Big|}\nolimits}

\newcommand{\CD}{\ensuremath C_{\gD}}

\newcommand{\Ieqref}[1]{\textup{\tagform@{I.\ref{I_#1}}}}
\newcommand{\Iref}[1]{I.\ref{I_#1}}
\newcommand{\IIeqref}[1]{\textup{\tagform@{II.\ref{II_#1}}}}
\newcommand{\IIref}[1]{II.\ref{II_#1}}

\newcommand{\Rbp}{\ensuremath\overline{\R}_+}

\newcommand{\Xil}[1][n-1]{\ensuremath\Xi^{\le #1}}
\newcommand{\cCl}[1][n-1]{\ensuremath\cC^{\le #1}}
\newcommand{\Ql}[1][n-1]{\ensuremath Q^{\le #1}}

\newcommand{\Xiu}[1][n]{\ensuremath\Xi^{\ge #1}}
\newcommand{\cCu}[1][n]{\ensuremath\cC^{\ge #1}}
\newcommand{\Qu}[1][n]{\ensuremath Q^{\ge #1}}
\newcommand{\bcB}{\ensuremath B}
\newcommand{\bRV}{\ensuremath B(\R_+\times V_c)}
\newcommand{\bcBG}[1][r]{\ensuremath B(\cG^#1)}
\newcommand{\Rrp}[1][r]{\ensuremath \hat \R^{#1}_+}
\newcommand{\sh}{\ensuremath\text{shad}}
\newlength{\BCs@ze}
\newlength{\BCsh@ft}
\ifcase\@ptsize\relax
\or
\or
\fi
\DeclareFixedFont\MT{OMS}{cmsy}{m}{n}{\BCs@ze}    
\newcommand{\BigCart}{\ensuremath\mathop{\raisebox{\BCsh@ft}{{\MT\char"02}}}}

\def\pdftitle{\@gobble}

\makeatother

\numberwithin{equation}{section}
\allowdisplaybreaks[2]

\date{December 6, 2010}

\title[Brownian Motions on Metric Graphs III]{%
Brownian Motions on Metric Graphs III\\
{\small Construction: General Metric Graphs}
}

\pdftitle{Brownian Motions on Metric Graphs III -
          Construction: General Metric Graphs}

\author[V.~Kostrykin]{Vadim Kostrykin}
\address{Vadim Kostrykin\newline
Institut f\"ur Mathematik\newline
Johannes Gutenberg--Universit\"at\newline
D--55099 Mainz, Germany}
\email{kostrykin@mathematik.uni-mainz.de}

\author[J.~Potthoff]{J\"urgen Potthoff}
\address{J\"urgen Potthoff\newline
Institut f\"ur Mathematik, Universit\"at Mannheim\newline
D--68131 Mann\-heim, Germany}
\email{potthoff@math.uni-mannheim.de}

\author[R.~Schrader]{Robert Schrader}
\address{Robert Schrader\newline
Institut f\"{u}r Theoretische Physik\newline
Freie Universit\"{a}t Berlin, Arnimallee~14\newline
D--14195 Berlin, Germany}
\email{schrader@physik.fu-berlin.de}

\copyrightinfo{2010}{V.~Kostrykin, J.~Potthoff, R.~Schrader}
\subjclass[2010]{60J65,60J45,60H99,58J65,35K05,05C99}
\keywords{Brownian motion, metric graph, Wentzell boundary condition}

\begin{document}
\begin{abstract}
Consider a metric graph $\cG$ with set of vertices $V$. Assume that for every vertex
in $V$ one is given a Wentzell boundary condition. It is shown how one can construct
the paths of a Brownian motion on $\cG$ such that its generator --- viewed as an
operator on the space $C_0(\cG)$ of continuous functions vanishing at infinity ---
has a domain consisting of twice continuously differentiable functions with second
derivative in $C_0(\cG)$, satisfying these boundary conditions.
\end{abstract}

\maketitle
\thispagestyle{empty}

\section{Introduction and Main Result} \label{sect1}
This is the third in a series of three articles about the characterization, the
construction and the basic properties of Brownian motions on metric graphs. In the
first of these articles~\cite{BMMG1}, Brownian motions on metric graphs have been
defined, their Feller property has been shown, and their generators have been
determined, i.e., the analogue of Feller's theorem for metric graphs has been
proved. In the second article~\cite{BMMG2}, all possible Brownian motions on single
vertex graphs were constructed. In the present article all possible Brownian
motions (as defined in~\cite{BMMG1}) on a general metric graph are constructed from
the trajectories of Brownian motions on single vertex graphs. In a companion
article~\cite{BMMG0}, which serves more as a background for this series and which
can also be read as an introduction to the topic, we revisit the classical
cases of Brownian motions on bounded intervals and on the semi-line
$\R_+$, see~\cite{Fe54, Fe54a, Fe57, ItMc63, ItMc74, Kn81}.

For a general introduction to the subject of this article we refer the interested
reader to~\cite{BMMG1}, henceforth quoted as ``article~I'', while~\cite{BMMG2} is
cited as ``article~II''. We shall refer to equations, definitions, theorems etc.\
from article~I or II by placing an ``I'' or ``II'', respectively, in front. For
example, ``formula~\Ieqref{eq_2_4}'' refers to formula~(\ref{I_eq_2_4}) in
article~I, while ``definition~\Iref{def_3_1}'' points to definition~\ref{I_def_3_1}
in article~I. Unless otherwise mentioned, we continue to use the notation and the
conventions set up in these articles.

Metric graphs arise as the underlying structure of models in many domains of
science, such as physics, chemistry, computer science and engineering to mention
just a few. We refer the interested reader to~\cite{Ku04} for a review of such
models and for further references. Mathematically, metric graphs are piecewise
linear spaces with singularities at the vertices of the graph. In the present series
of articles we exclusively consider \emph{finite} metric graphs. Heuristically
speaking a finite metric graph may be viewed as a finite collection of compact or
semi-infinite intervals --- the edges --- with some of their endpoints --- the
vertices --- identified. The standard metric on the real line induces thereon a
metric in a natural way. A formal definition can be found, e.g., in article~I, and
the most important notions are also given in section~\ref{sect2} below.

We quickly recall in a slightly informal way some definitions and the main result of
article~I.

Let $\cG$ be a finite metric graph with a set of vertices $V$ and a set of edges denoted by
$\cL$. A \emph{Brownian motion on $\cG$} is by definition any normal strong Markov
process $X$ on $\cG\cup\{\gD\}$, $\gD$ being a cemetery state, such that: (i) its
paths are right continuous with left limits in $\cG$, and which are continuous up to
the lifetime $\zeta$ of $X$; (ii) when $X$ starts on an edge $l\in\cL$, $l$ isomorphic
to $[a,b]$ or to $[0,+\infty)$, and $X$ is stopped when hitting a vertex of $\cG$
(to which $l$ necessarily is incident), then the stopped process is equivalent to a
standard Brownian motion on $[a,b]$ or $[0,+\infty)$, respectively, with absorption
in the endpoint(s) of the interval. For a formal definition the reader is referred
to definition~\Iref{def_3_1}. We want to emphasize that our definition (which is a
generalization of the one given by Knight~\cite{Kn81}) excludes any jumps of $X$
other than those from a vertex to the cemetery point.

The Banach space of real valued, continuous functions on $\cG$ vanishing at infinity,
equipped with the sup-norm, is denoted by $C_0(\cG)$. Consider the generator $A$ of
$X$ on $C_0(\cG)$ with domain $\cD(A)$. Define the space $\Cii$ to consist of those
functions $f$ in $C_0(\cG)$ which are twice continuously differentiable in the open
interior $\cG^\circ = \cG\setminus V$ of $\cG$, and which are such that their second
derivative $f''$ extends from $\cG^\circ$ to a function in $C_0(\cG)$. (We identify
every metric graph with its \emph{geometric graph}, see, e.g., \cite{Ju05}, and
thereby we may consider the set of vertices and every edge as a subset of the set
$\cG$.) Let $V_\cL$ denote the subset of $V\times\cL$ given by
\begin{equation*}
    V_\cL = \bigl\{(v,l),\,v\in V \text{ and }l\in\cL(v)\bigr\}.
\end{equation*}
$\cL(v)$ stands for the set of edges of $\cG$ which are incident with $v$. We shall
also write $v_l$ for $(v,l)\in V_\cL$. Moreover, if $f$ is a real valued function on
$\cG$, $l\in\cL$ and $v\in V$ such that $l$ is incident with $v$, then the
\emph{directional derivative $f'(v_l)$ of $f$ at $v$ in direction $l$} is defined as
the inward normal derivative of $f$ on $l$ at $v$ whenever it exists. Consider data
of the following form
\begin{equation}    \label{eq1i}
\begin{split}
    a &= (a_v,\,v\in V)\in [0,1)^V\\
    b &= (b_{v_l},\,v_l\in V_\cL) \in [0,1]^{V_\cL}\\
    c &= (c_v,\,v\in V)\in [0,1]^V
\end{split}
\end{equation}
subject to the condition
\begin{equation}    \label{eq1ii}
    a_v + \sum_{l\in \cL(v)} b_{v_l} + c_v =1,\qquad \text{for every $v\in V$}.
\end{equation}
Define a subspace $\cH_{a,b,c}$ of $\Cii$ as the space of those functions
$f$ in $\Cii$ which at every vertex $v\in V$ satisfy the \emph{Wentzell boundary
condition}
\begin{equation}\label{eq1iii}
    a_v f(v) - \sum_{l\in\cL(v)} b_{v_l}f'(v_l)+ \frac{1}{2}\,c_v f''(v)=0.
\end{equation}
The main result of article~I (theorem~\Iref{thm_5_3}) is the

\begin{theorem}[Feller's theorem for metric graphs] \label{thm1i}
Let $X$ be a Brownian motion on $\cG$, and let $A$ be its generator on $C_0(\cG)$
with domain $\cD(A)$. Then there are $a$, $b$, $c$ as in~\eqref{eq1i}, \eqref{eq1ii},
so that $\cD(A)=\cH_{a,b,c}$.
\end{theorem}

The main result of the present article is the converse statement, namely

\begin{theorem} \label{thm1ii}
For any choice of the data as in~\eqref{eq1i}, \eqref{eq1ii}, there is a Brownian
motion $X$ in the sense of~\cite{BMMG1} on the metric graph $\cG$ so that its
generator $A$ has $\cH_{a,b,c}$ as its domain.
\end{theorem}

The proof we give in this article consists in the pathwise construction of a
Brownian motion for any given set of data $a$, $b$, $c$. It will be carried out in
section~\ref{sect3}. The basic strategy is laid out in section~\ref{sect2}: There we
show how one can construct from two metric graphs a new metric graph by joining a
certain number of external edges of the given graphs. Furthermore, if one is given
two Brownian motions (in the sense of article~I) on the given graphs, then on the
new graph a Brownian motion is constructed from the previous ones by pasting their
trajectories together appropriately.

The article is concluded in section~\ref{sect4} by a discussion of the inclusion
of tadpoles. Furthermore, there is an appendix with a technical result on
the crossover times which is used in section~2.

\section{Joining Two Metric Graphs} \label{sect2}
Throughout this section we suppose that $\cG_k=(V_k,\cI_k, \cE_k,\p_k)$, $k=1$, $2$,
are two finite metric graphs. That is, for $k=1$, $2$, $V_k$ is a nonempty finite
set of vertices, $\cI_k$ a finite set of internal edges, each of which is isomorphic
to a compact interval, $\cE_k\ne\emptyset$ is a finite set of external edges, each
of which is isomorphic to the half line $\R_+$. $\p_k$, $k=1$, $2$, is a mapping
from $\cI_k \cup\cE_k$ into $V_k\cup(V_k\times V_k)$ which determines the
combinatorial structure of the graph $\cG_k$. It maps an internal edge $i\in\cI_k$
to an ordered pair $(\p_k^-(i),\p_k^+(i))\in V_k\times V_k$ of vertices, called the
\emph{initial} and \emph{final vertex of $i$}, while $e\in\cE_k$ is mapped to
$\p_k(e)\in V_k$, called the \emph{initial vertex of $e$}. (For a more detailed
description of metric graphs, the reader is referred to article~I and the references
quoted there.) In the following subsection we shall construct a new metric graph
$\cG=(V,\cI, \cE,\p)$ from $\cG_1$ and $\cG_2$ by connecting some of their external
edges.

Except for our discussion in section~\ref{sect4}, we shall assume throughout that
the metric graphs under consideration do not have \emph{tadpoles}, i.e., internal
edges $i$ for which initial and final vertices coincide.

As in~\cite{BMMG1}, we shall identify each metric graph with its \emph{geometric graph}
(see, e.g., \cite{Ju05}). That is, we also consider a metric graph as the union
of a collection of compact intervals and half-lines, endowed with an equivalence
relation identifying some the endpoints of these intervals, which determines the
combinatorial structure of the graph. In this sense, we view the set of vertices and
the edges as subsets of the graph, and their points as points of the graph.

It will be convenient to consider the metric graphs $\cG_1$, $\cG_2$ as
subgraphs of the metric graph $\cG_0=\cG_1\uplus\cG_2$  which is their (disjoint)
union: $\cG_0=(V_0,\cI_0,\cE_0,\p_0)$, with $V_0=V_1\cup V_2$, $\cI_0=\cI_1\cup
\cI_2$, $\cE_0=\cE_1\cup \cE_2$, and where the map $\p_0$ comprises the maps
$\p_1$, $\p_2$ in the obvious way.

\subsection{Construction of the graph \boldmath $\cG$} \label{ssect2i}
Suppose that $N$ is a natural number such that $N\le \min(|\cE_1|,|\cE_2|)$. For
$k=1$, $2$, select subsets $\cE_k'\subset \cE_k$ of edges with
$|\cE_1'|=|\cE_2'|=N$ to be joined. Let these sets be labeled as follows
\begin{equation*}
	\cE'_1 = \{e_1,\dotsc,e_N\},\quad
	\cE'_2 = \{l_1,\dotsc,l_N\}.
\end{equation*}
\begin{figure}[ht]
\begin{center}
    \includegraphics[scale=.8]{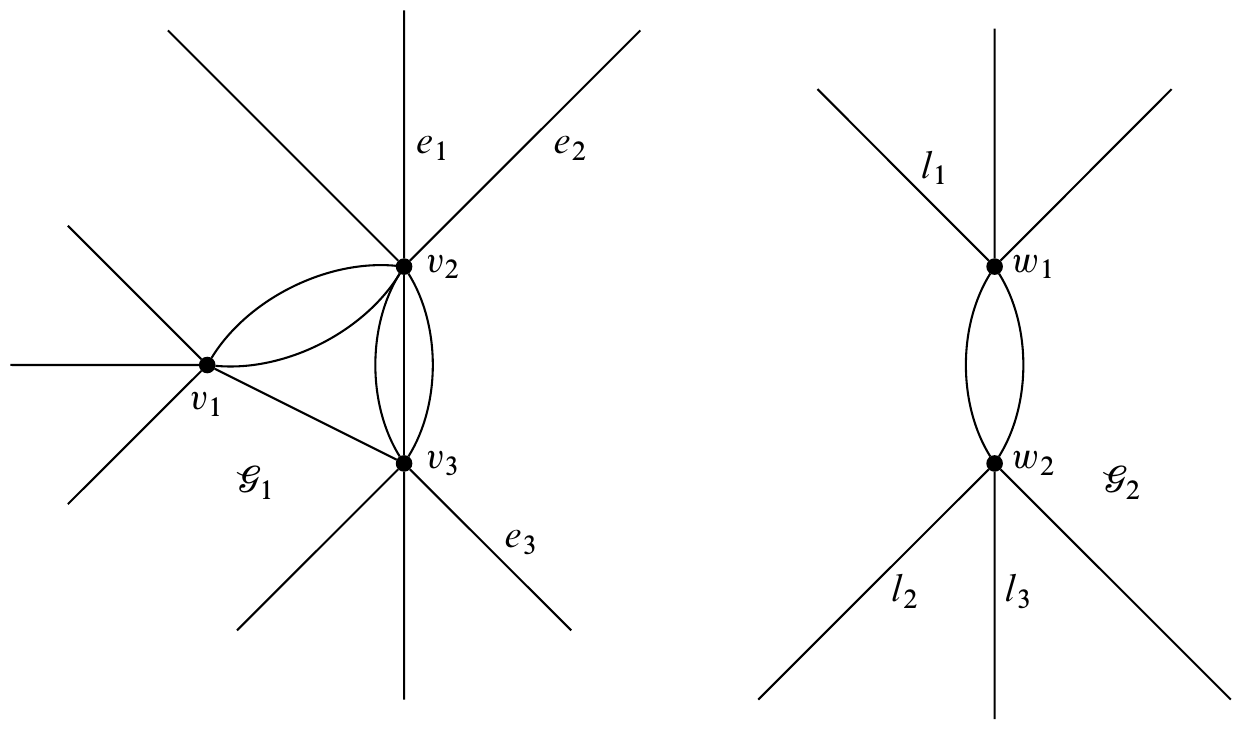}
    \caption{Two metric graphs $\cG_1$, $\cG_2$, to be joined by connecting the
             pairs of external lines
             $(e_1,l_1)$, $(e_2,l_2)$ and $(e_3,l_3)$.} \label{fig1}
\end{center}
\end{figure}

In addition we assume that we are given strictly positive numbers $b_1$, \dots, $b_N$,
which will serve as the lengths of the new internal edges, as well as $\gs_k\in\{-1,1\}$,
$k=1$, \dots, $N$, which will determine the orientations of the new internal edges.
For every $k\in\{1,\dotsc,N\}$ we associate with the interval $[0,b_k]$ an abstract
edge $i_k$ (not in $\cI_0$) which is isomorphic to $[0,b_k]$. Set $\cI_c =
\{i_1,\dotsc,i_N\}$, and
\begin{align*}
	V   &= V_0,\\
	\cI &= \cI_0\cup\cI_c,\\
	\cE &= \cE_0\setminus (\cE'_1\cup\cE'_2).
\end{align*}
The combinatorial structure of $\cG$ is determined by $\p$, which we construct in
two steps: Let $\p'$ be the restriction of $\p_0$ to
$\cI_0\cup\cE_0\setminus(\cE'_1\cup\cE'_2)$. Then $\p$ is the extension of $\p'$ to
$\cI\cup\cE$, which is defined by
\begin{equation*}
	\p(i_k) = \begin{cases}
				\bigl(\p_1(e_k),\p_2(l_k)\bigr),	&\text{if $\gs_k=1$},\\
				\bigl(\p_2(l_k),\p_1(e_k)\bigr),	&\text{if $\gs_k=-1$},
			 \end{cases}
                \qquad k=1,\dotsc,N.
\end{equation*}
\begin{figure}[ht]
\begin{center}
    \includegraphics[scale=.8]{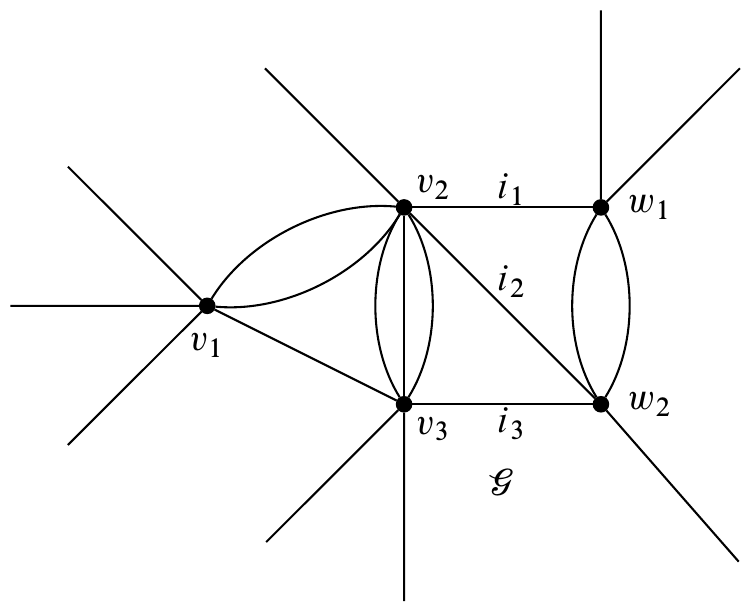}
    \caption{The graph $\cG$ after joining $\cG_1$ and $\cG_2$.}  \label{fig2}
\end{center}
\end{figure}
Figure~\ref{fig2} shows an example of a metric graph which is constructed from the
two metric graphs in figure~\ref{fig1} by joining the $N=3$ pairs of external edges
$(e_1,l_1)$, $(e_2,l_2)$ and $(e_3,l_3)$. The new internal edges $i_1$, $i_2$ and
$i_3$ have the lengths $1$, $\sqrt{2}$ and $1$ respectively (in some scale).

Conversely, let a metric graph $\cG$ be given. Associate with every vertex $v\in V$
of $\cG$ a single vertex graph $\cG(v)$ with vertex $v$ and $n(v)$ external
edges, where $n(v)$ is the number of edges incident with $v$ in $\cG$. Then it is clear
that we can reconstruct $\cG$ from the single vertex graphs $\cG(v)$, $v\in V$,
by finitely many applications of the joining procedure described above.

For the purposes below it will be convenient to introduce some additional notation.
We let $V_c\subset V$ denote the subset of vertices of $\cG$ which are connected to
each other by the new internal edges in $\cI_c$. That is, $v\in V_c$ is such that
there exists at least one $i\in\cI_c$ with $v\in\p(i)$. For notational simplicity,
here and below we also use $\p(l)$ to denote the set consisting of $\p^-(l)$ and
$\p^+(l)$ if $l\in \cI$, and of $\p(l)$ if $l\in\cE$. In the example of the
figures~\ref{fig1} and~\ref{fig2}, $V_c = \{v_2,v_3,w_1,w_2\}$.

Recall that for a metric graph $\cH = (V_\cH,\cI_\cH,\cE_\cH,\p_\cH)$, $\cH^\circ =
\cH\setminus V_\cH$ denotes its open interior. Similarly, if $l$ is an edge in
$\cI_\cH\cup\cE_\cH$, then  its open interior $l^\circ$ is defined to be the set
$l\setminus\p_\cH(l)$. Any point $\xi\in\cH^\circ$ is in one-to-one correspondence
with its \emph{local coordinates} $(l,x)$, where $l$ is the edge to which $\xi$
belongs, and $x\in(0,+\infty)$ if $l\in\cE_\cH$, while $x\in (a,b)$ if $l\in\cI_\cH$
is isomorphic to $[a,b]$.

Consider a vertex $v\in V_c$ which belongs to $\cG_1$, and let $i_k\in\cI_c$,
$k\in\{1,\dotsc,N\}$, be an internal edge connecting $v$ to $\cG_2$, i.e.,
$v\in\p(i_k)$. Then the point $\eta\in\cG_2^\circ$ with local coordinates
$(l_k,b_k)$ is called a \emph{shadow vertex} of the vertex $v$. $\sh(v)\subset \cG^0_2$
is the set of all shadow vertices of $v$. If $v\in V_c\cap\cG_2$, its set of shadow
vertices (which are points in $\cG_1^\circ$) is defined analogously. $V_s=\sh(V_c)
= \cup_{v\in V_c}\sh(v)$ is the set of all shadow vertices. If $\xi\in V_s$, then
there exists a unique $v\in V_c$ so that $\xi\in\sh(v)$. Thus, setting $\gk(\xi) = v$ we
have defined a mapping $\gk$ from $V_s$ onto $V_c$. Of course, in general $\gk$ is not
injective. In figure~\ref{fig3} the shadow vertices of the example above are
depicted as small circles on the external edges, i.e., $V_s=\{\xi_1,\xi_2,\xi_3,\eta_1,
\eta_2,\eta_3\}$. For example, $\sh(v_2)=\{\eta_1,\eta_2\}$, $\sh(w_1)=\{\xi_1\}$,
and $\gk(\xi_2)=w_2$, $\gk(\eta_2)=v_2$.
\begin{figure}[ht]
\begin{center}
    \includegraphics[scale=.8]{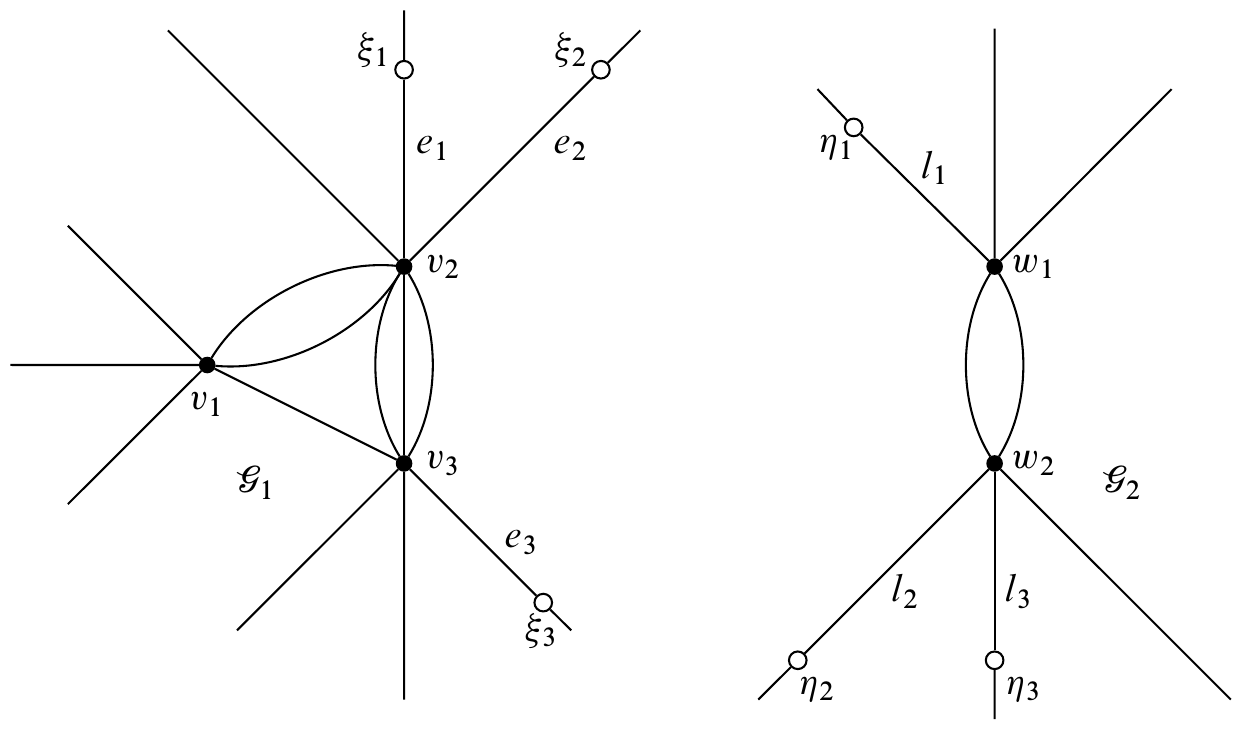}
    \caption{The graphs $\cG_1$, $\cG_2$ with the shadow vertices (small circles).}
    \label{fig3}
\end{center}
\end{figure}

\subsection{Construction of a Preliminary Version of the Brownian Motion} \label{ssect2ii}
From now we suppose that we are given a family of probability spaces
\begin{equation*}
    (\Xi^0,\cC^0,Q^0_\xi), \qquad \,\xi\in\cG_0,
\end{equation*}
and that thereon a Brownian motion with state space $\cG_0$ in the sense
of definition~\Iref{def_3_1} is defined. This Brownian motion is
denoted by $Z^0=(Z^0(t),\,t\in\R_+)$. Actually, since $\cG_0=\cG_1\cup \cG_2$ and
$\cG_1$, $\cG_2$ are disconnected, this is the same as saying that we are given a
Brownian motion on $\cG_1$ and one on $\cG_2$. However, notationally it will be more
convenient to view this as one stochastic process. We assume, as we may, that $Z^0$
has exclusively c\`adl\`ag paths which are continuous up to the lifetime $\zeta^0$ of $Z^0$.
As in the previous articles, $\gD$ denotes a universal cemetery point.
$\cF^0=(\cF^0_t,\,t\in\R_+)$ denotes the natural filtration of $Z^0$. The hitting
time of $V_s$ by $Z^0$ is denoted by $\tau^0$, i.e.,
\begin{equation*}
    \tau^0 = \inf\,\{t>0,\,Z^0(t) \in V_s\}.
\end{equation*}
Furthermore, we assume that $\vt=(\vt_t,\,t\in\R_+)$ is a family of shift operators
for $Z^0$ acting on $\Xi^0$.

For any topological space $(T,\cT)$ denote by $C_\gD(\R_+,T)$ the space of mappings
$f$ from $\R_+$ into $T\cup\{\gD\}$ which are right continuous, have left limits in
$T$, are continuous up to their lifetime
\begin{equation*}
    \zeta_f = \inf\,\{t>0,\,f(t)= \gD\},
\end{equation*}
and which are such that $f(t)=\gD$ implies $f(s)=\gD$ for all $s\ge t$. In particular
and in the present context, $f\in\CD(\R_+,\cG_0)$ is either continuous from $\R_+$ into
$\cG_0$ or it has a jump from $\cG_1$ or $\cG_2$ to $\gD$, but there can be no jump
from $\cG_1$ to $\cG_2$ or vice versa.

We shall make use of some special versions of the process $Z^0$, which we introduce
now. For every $v\in V_c$, $Z^1_v=(Z^1_v(t),\,t\in\R_+)$ denotes a Brownian motion
on $\cG_0$ defined on another probability space $(\Xi^1_v,\cC^1_v,\mu^1_v)$. We
suppose that $Z^1_v$ exclusively has paths which start in $v$ and which belong to
$\CD(\R_+,\cG_0)$. (For example, one can use a standard path space construction to
obtain such a version from $(\Xi^0,\cC^0, Q^0_v, Z^0)$.) The hitting time of $V_s$ by
$Z^1_v$ is denoted by $\tau^1_v$, its lifetime by $\zeta^1_v$.

The idea to define the preliminary version $Y=(Y(t),\,t\in\R_+)$ of the Brownian
motion on $\cG$ is to construct its paths as follows. Let $\xi\in\cG$ be a given
starting point. $\cG$ (viewed as a set) has the following decomposition (cf.\
figure~\ref{fig4}):
\begin{equation*}
	\cG =  \hat\cG_1\uplus\hat\cG_2,
\end{equation*}
with
\begin{align*}
	\hat\cG_1 &= \cG_1\setminus \bigl(e^\circ_1\cup\dotsb\cup e^\circ_N\bigr),\\
    \hat\cG_2 &= \bigl(\cG_2\setminus \bigl(l^\circ_1\cup\dotsb\cup l^\circ_N\bigr)\bigr)
                        \cup \bigl(i_1^\circ\cup\dotsc\cup i_N^\circ\bigr).
\end{align*}
\begin{figure}[ht]
\begin{center}
    \includegraphics[scale=.8]{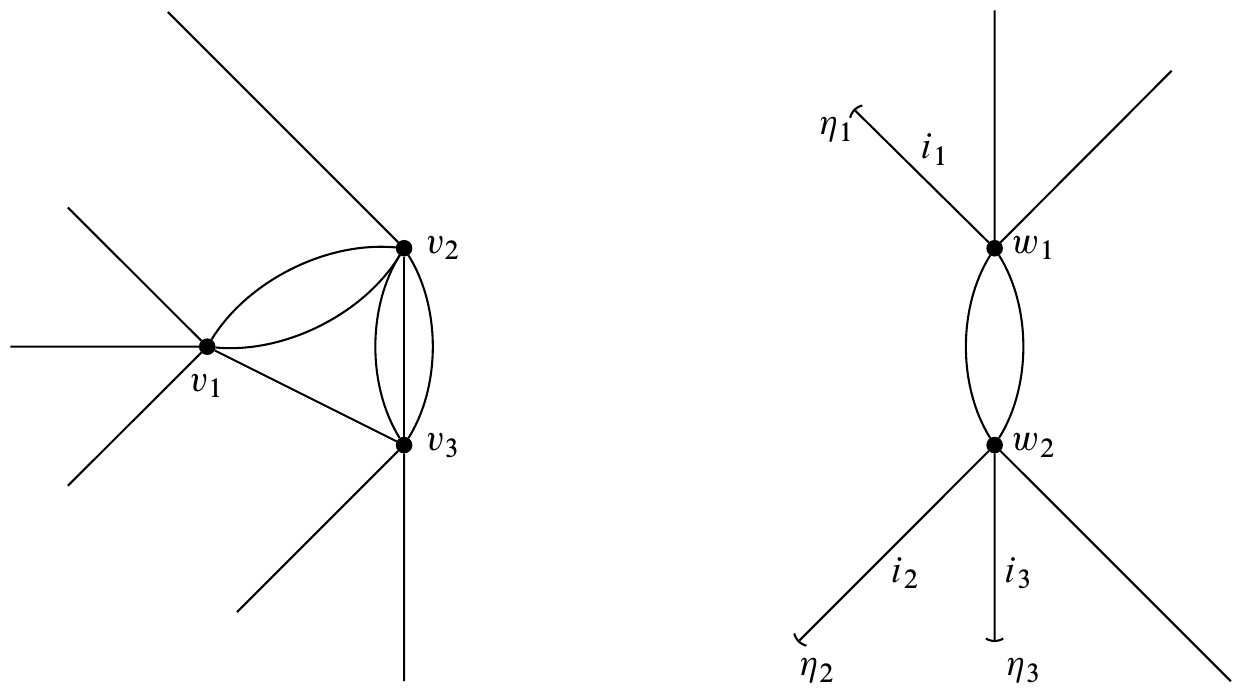}
    \caption{The starting points of $Y$.} \label{fig4}
\end{center}
\end{figure}
Thus we may consider $\xi$ instead as a point in $\hat\cG_1\uplus\hat\cG_2\subset
\cG_0$.

We pause here for the following remark: Of course, the convention we make that
all new open inner edges $i^\circ_1,\dotsc,i^\circ_N$ are attached to $\hat\cG_2$
is somewhat arbitrary. Just as well any subset of them could have been attached
to $\hat\cG_1$ instead. Even though different conventions lead to processes
with different paths, the main result of this section, theorem~\ref{thm2xiii},
remains unchanged. It follows that all resulting processes are equivalent to each
other.

Let $Y$ start as $Z^0$ in $\xi\in\hat\cG_1\uplus\hat\cG_2$, and consider one trajectory.
(In order to avoid any confusion, let us point out that even though $\xi\in\hat\cG_k$,
$k=1$, $2$, the process $Z^0$ moves in $\cG_k$.) If this
trajectory reaches the cemetery point $\gD$ before hitting the set $V_s$ of shadow
vertices, it is the complete trajectory of $Y$ and it stays forever at the cemetery.
If the trajectory hits a shadow vertex $\eta\in V_s$ before its lifetime expires,
this piece of the trajectory of $Y$ ends at the hitting time $\tau^0$. Set
$v=\gk(\eta)$, and let the trajectory of $Y$ continue with an (independent) trajectory
of $Z^1_v$ until its lifetime expires or it hits a shadow vertex, and so on.
Figure~\ref{fig5} explains the idea.
\begin{figure}[ht]
\begin{center}
    \includegraphics[scale=.8]{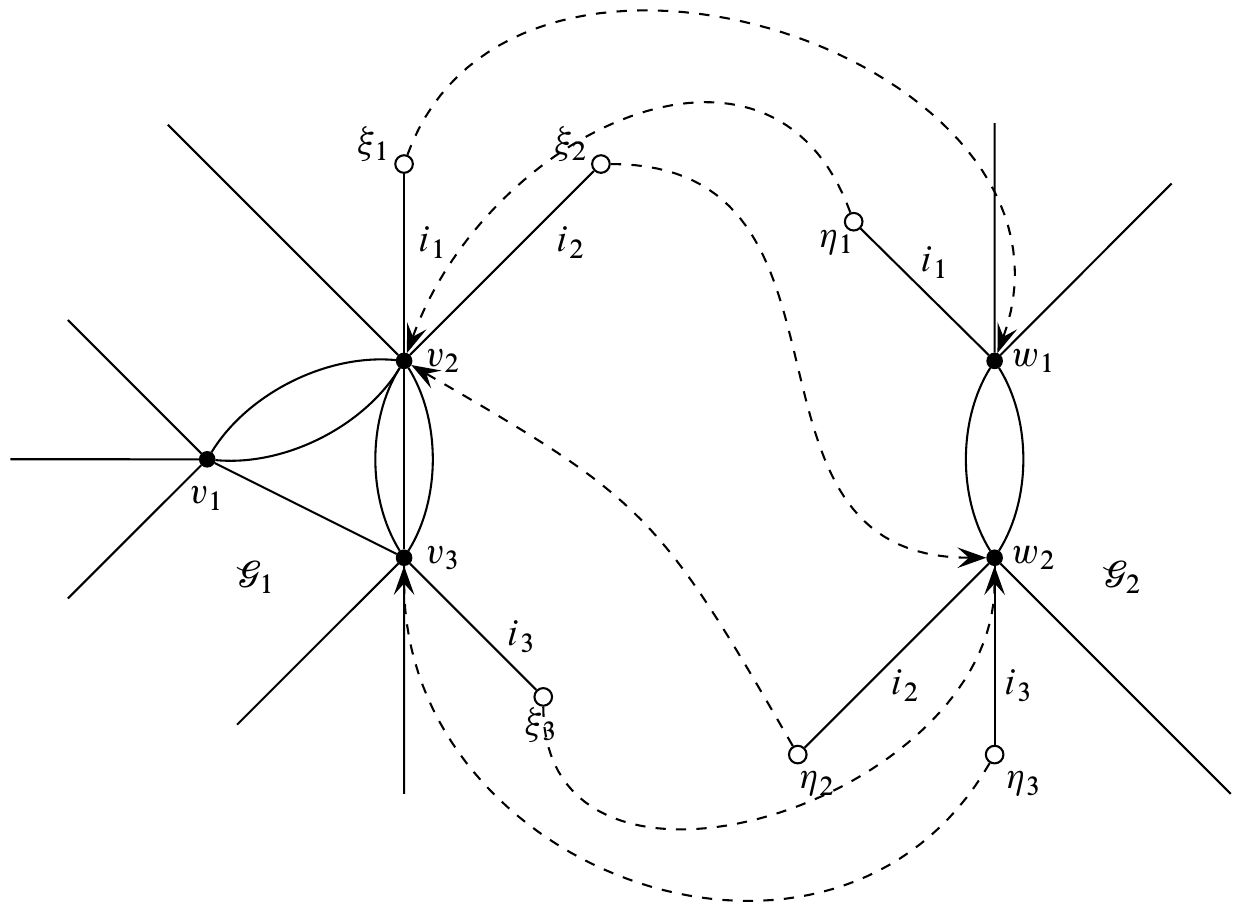}
    \caption{The construction of the process $Y$.} \label{fig5}
\end{center}
\end{figure}

This construction is formalized in the following way. Define
\begin{align*}
    \Xi^1 &= \BigCart_{v\in V_c} \Xi^1_v\\
    \cC^1 &= \bigotimes_{v\in V_c} \cC^1_v\\
    Q^1   &= \bigotimes_{v\in V_c} \mu^1_v,
\end{align*}
and view each of the stochastic processes $Z^1_v$, $v\in V_c$, as well as the random
variables $\tau^1_v$, $\zeta^1_v$, as defined on this product space. Let
\begin{equation*}
    \bigl(\Xi^n,\cC^n,Q^n,Z^n,\tau^n,\zeta^n\bigr),\qquad n\in\N,\,n\ge 2,
\end{equation*}
be a sequence of independent copies of
\begin{equation*}
    \bigl(\Xi^1,\cC^1,Q^1,Z^1,\tau^1,\zeta^1\bigr),
\end{equation*}
where $Z^1=(Z^1_v,\,v\in V_c)$ and similarly for $\tau^1$, $\zeta^1$. Next set
\begin{align*}
    \Xi   &= \BigCart_{n=0}^\infty \Xi^n\\
    \cC   &= \bigotimes_{n=0}^\infty \cC^n\\
    Q_\xi &= Q^0_\xi \otimes\Bigl(\bigotimes_{n=1}^\infty Q^n\Bigr),\qquad \xi\in\cG.
\end{align*}

The procedure sketched above of pasting together pieces of the trajectories of the
various processes $Z^n_v$ is controlled by a Markov chain $(K_n,\,n\in\N)$ which
moves at random times $(S_n,\,n\in\N)$ in the state space $V_c\cup\{\gD\}$. We
set out to construct this chain $\bigl((S_n,K_n),\,n\in\N\bigr)$. Define $S_1=\tau^0$. On
$\{S_1=+\infty\}$, i.e., in the case when $\zeta^0<\tau^0$, set $K_1 = \gD$.
Otherwise define
\begin{equation*}
    K_1 = \gk\bigl(Z^0(\tau^0)\bigr).
\end{equation*}
Observe that since all processes considered have right continuous paths, they are all
measurable stochastic processes, and therefore the evaluation of their time argument at
a random time yields a well-defined random variable. Set $S_2=+\infty$ on
$\{S_1=+\infty\}$, while
\begin{equation*}
    S_2 = S_1 + \tau^1_{K_1}
\end{equation*}
on $\{S_1<+\infty\}$. On $\{S_2=+\infty\}$ put $K_2=\gD$, and on its complement
\begin{equation*}
    K_2 = \gk\bigl(Z^1_{K_1}(\tau^1_{K_1})\bigr).
\end{equation*}
These construction steps are iterated in the obvious way: The
sequence
\begin{equation*}
    \bigl((S_n,K_n),\,n\in\N\bigr)
\end{equation*}
is inductively defined by $S_n=+\infty$ and $K_n=\gD$ on $\{S_{n-1}=+\infty\}$,
while
\begin{align*}
	S_n &= S_{n-1} + \tau^{n-1}_{K_{n-1}},\\[1ex]
	K_n &= \gk\bigl(Z^{n-1}_{K_{n-1}}(\tau^{n-1}_{K_{n-1}})\bigr)
\end{align*}
on $\{S_{n-1}<+\infty\}$.

Note that by construction $K_n = \gD$, $n\in\N$, if and only if $S_n=+\infty$,
and in that case $K_{n'}=\gD$, $S_{n'}=+\infty$ for all $n'\ge n$. Thus
$(+\infty,\gD)$ is a cemetery state for the chain $((S_n,K_n),\,n\in\N)$.

For example with a Borel--Cantelli argument it is not hard to see (cf.\
also~\cite{BMMG0}) that there exists a set $\Xi'\in\cC$ so that for all $\xi\in\cG$,
$Q_\xi(\Xi')=0$, and for all $\go\in\Xi\setminus\Xi'$  the sequence $(S_n(\go),\,n\in\N)$
increases to $+\infty$ in such a way that for all $n\in\N$, $S_n(\go)<S_{n+1}(\go)$ holds
when $S_n(\go)<+\infty$.

Now we are ready to construct $Y=(Y(t),\in\R_+)$. Let $\xi\in\cG =
\hat\cG_1\uplus\hat\cG_2$ be a given starting point, and suppose that $t\in\R_+$ is
given. On $\Xi'$ set $Y(t)=\gD$. On $\Xi\setminus\Xi'$ there is a unique $n\in\N_0$
so that $t\in [S_n, S_{n+1})$, with the convention $S_0=0$. If $t\in [0,S_1)$,
define $Y(t) = Z^0(t)$. If $t\in [S_n, S_{n+1})$ for $n\in\N$, then necessarily
$S_n$ is finite, so that $K_n\in V_c$, and we define
\begin{equation*}
	Y(t) = Z^n_{K_n}(t-S_n).
\end{equation*}
In addition, we make the convention $Y(+\infty)=\gD$. The natural
filtration generated by $Y$ will be denoted by $\cF^Y=(\cF^Y_t,\,t\in\R_+)$.

It follows from the construction of $Y$ that $\gD$ is a cemetery state for $Y$.
Indeed, suppose that $\go\in\Xi\setminus \Xi'$, and that the trajectory
$Y(\,\cdot\,,\go)$ reaches the point $\gD$ at a finite time $\zeta^Y(\go)$. This
implies that there is an $n\in\N_0$ such that $\zeta^Y(\go)\in
[S_n(\go),S_{n+1}(\go))$. Then $S_n(\go)$ is finite, and therefore $K_n(\go)\in V_c$
so that $Y(\,\cdot\,,\go)$ is equal to $Z^n_{K_n}(\,\cdot\,-S_n(\go),\go)$ on the
interval $[S_n(\go),S_{n+1}(\go))$, and this trajectory reaches $\gD$ before hitting
a shadow vertex. Hence $\tau^n_{K_n}=+\infty$ which entails that
$S_{n+1}(\go)=+\infty$. Consequently after $S_n(\go)$ there are no finite
crossover times for this trajectory, and therefore $Y(\,\cdot\,,\go)$ stays
at~$\gD$ forever. Furthermore note that the left limit $Y(\zeta^Y(\go)-,\go)$
at $\zeta^Y(\go)$ belongs to $V_0$.

In terms of the stochastic process $Y$ the random times $S_n$, $n\in\N$, have the
following description. Suppose that $Y$ starts in $\xi\in\cG$. Then $S_1$ is the
hitting time of $V_c$. But if $\xi\in V_c$, then actually it is the hitting time of
$V_c\setminus\{\xi\}$, because it hits a vertex in $V_c$ which corresponds to the
first hitting of a shadow vertex, i.e., a point in $\cG_0$ different from $\xi$,
by $Z^0$. In particular, $S_1>0$. Similarly, $S_n$ is the hitting time of
$V_c\setminus\{K_{n-1}\}$ by $Y$ after time $S_{n-1}$. In appendix~\ref{appA} it
is shown that for every $n\in\N$, $S_n$ is a stopping time with respect to $\cF^Y$.

It follows from its construction that $Y$ is a normal process,
that is, for every $\xi\in\cG$, $Q_\xi(Y(t=0)=\xi)=1$. Furthermore, all paths of $Y$ belong
to $\CD(\R_+,\cG)$. Let $S_V$ be the hitting time of the set of vertices $V$ of
$\cG$ by $Y$. Then $S_V \le S_1$, because $V_c\subset V$ and therefore we find that
$Y(\,\cdot\,\land S_V)$ is pathwise equal to $Z^0(\,\cdot\,\land S^0_V)$, where
$S^0_V$ denotes the hitting time of $V$ by $Z^0$. Suppose that the starting point
$\xi$ belongs to $l^\circ$, $l\in\cI\cup\cE$, and $l$ is isomorphic to the interval
$I$. Then by definition of $Z^0$ (cf.\ definition~\Iref{def_3_1}), the stopped
process $Z^0(\,\cdot\,\land S^0_V)$ is equivalent to a standard Brownian motion on the
interval $I$ with absorption at the endpoint(s) of $I$. Hence the same is true for
$Y$: $Y(\,\cdot\,\land S_V)$ is equivalent to a standard Brownian motion on $I$
with absorption at the endpoint(s) of $I$.

\subsection{Markov property of \boldmath $Y$} \label{ssect2iii}
For any measurable space $(M,\cM)$, $B(M)$ denotes the space of bounded, measurable
functions on $M$. Every $f\in B(\cG^n)$, $n\in\N$, is extended to
$(\cG\cup\{\gD\})^n$ by $f(\xi_1,\dotsc,\xi_n)=0$,
$(\xi_1,\dotsc,\xi_n)\in(\cG\cup\{\gD\})^n$, whenever there is an index
$k\in\{1,\dotsc, n\}$ so that $\xi_k=\gD$. In this subsection we shall prove the
following

\begin{proposition} \label{prop2i}
$Y$ has the simple Markov property: For all $f\in B(\cG)$, $s$, $t\in\R_+$,
$\xi\in\cG$,
\begin{equation}    \label{eq2i}
    E_\xi\bigl(f\bigl(Y(s+t)\bigr)\cond \cF^Y_s\bigr)
        = E_{Y(s)}\bigl(f\bigl(Y(t)\bigr)\bigr)
\end{equation}
holds true $Q_\xi$--a.s.\ on $\{Y(s)\ne\gD\}$.
\end{proposition}

The proof of proposition~\ref{prop2i} is somewhat technical and lengthy.
Therefore it will be broken up into a sequence of lemmas.

For every $n\in\N$ the probability space $(\Xi,\cC,Q_\xi)$, $\xi\in\cG$, underlying
the construction of the process $Y$ may be written as the product of the probability
spaces $(\Xil, \cCl, \Ql_\xi)$ and $(\Xiu, \cCu, \Qu)$ with
\begin{align*}
	\Xil &= \BigCart_{j=0}^{n-1} \Xi^j,  &
				\Xiu &= \BigCart_{j=n}^{\infty} \Xi^j,\\
	\cCl &= \bigotimes_{j=0}^{n-1} \cC^j, &
				\cCu &= \bigotimes_{j=n}^{\infty} \cC^j,\\
	\Ql_\xi &= Q_\xi^0\otimes\Bigl(\bigotimes_{j=1}^{n-1} Q^j\Bigr), &
				\Qu &= \bigotimes_{j=n}^{\infty} Q^j.
\end{align*}
Introduce a family $\cB=(\cB_n,\,n\in\N_0)$ of sub--$\gs$--algebras of $\cC$ by
setting
\begin{equation*}
    \cB_n = \cC^{\le n-1}\times \Xi^{\ge n}.
\end{equation*}
Obviously, the family $\cB$ forms a filtration. Furthermore, from the construction
of $K_n$ and $S_n$ it is easy to see that the chain $((S_n, K_n),\,n\in\N)$ is
adapted to $\cB$.

First we study the chain $((S_n,K_n),\,n\in\N)$ in more detail. Recall our
convention that $S_0=0$. We set $\cB_0=\{\emptyset,\Xi\}$, and under the law
$Q_v$, $v\in V_c$, we put $K_0=v$. $g\in\bRV$ is extended to $\Rbp\times
\bigl(V_c\cup\{\gD\}\bigr)$ by $g(+\infty,\,\cdot\,)=g(\,\cdot\,,\gD)
=g(+\infty,\gD)=0$. For $g\in\bRV$, $n\in\N_0$, define
\begin{equation}    \label{eq2ii}
    (U_n g)(s,v) = E_{v}\bigl(g(s+S_n, K_n)\bigr),\qquad s\in\R_+,\,v\in V_c.
\end{equation}
Note that $U_0=\text{id}$, and that for every $g\in\bRV$ and all $n\in\N$,
$U_n g\in\bRV$. In particular, the convention mentioned above applies to $U_n g$,
too.

\begin{lemma}   \label{lem2ii}
{\ }
\begin{enum_a}
    \item For all $m$, $n\in\N$, $m\le n$,  $\xi\in\cG$, $s\ge 0$, and every
        $g\in\bRV$ the following formula holds true $Q_\xi$--a.s.\
            \begin{equation}    \label{eq2iii}
                E_\xi\bigl(g(s+S_n, K_n) \cond \cB_m)\bigr)
                    = (U_{n-m}g)(s+S_m,K_m).
            \end{equation}
    \item $(U_n,\,n\in\N_0)$ forms a semigroup of linear maps on $\bRV$. In
        particular, for all $(s,v)\in\R_+\times V_c$ under $Q_{v}$ the chain
        $\bigl((s+S_n, K_n),\,n\in\N_0\bigr)$ is a homogeneous Markov chain with
        transition kernel
            \begin{equation*}
                P\bigl((s,v), A\bigr)
                    = Q_{v}\bigl((s+S_1,K_1)\in A\bigr),\qquad A\in\cB(\R_+\times V_c).
            \end{equation*}
\end{enum_a}
\end{lemma}

\begin{proof}
For $m=n$ formula~\eqref{eq2iii} is trivial. Consider the case when $n\ge 2$, $m=n-1$.
Let $\gL\in\cB_{n-1}$, $v\in V_c$, and put $\gL_v = \gL\cap\{K_{n-1}=v\}\in\cB_{n-1}$.
From the construction of $S_n$ and $K_n$
\begin{align*}
    E_\xi\bigl(&g(s+S_n,K_n); \gL_v\bigr)\\[1ex]
        &\hspace{-1.5em}= \int_{\Xil[n-2]} 1_{\gL_v}
            \Bigl(\int_{\Xiu[n-1]}g\bigl(s+S_{n-1}+\tau^{n-1}_v,
                    \gk\bigl(Z^{n-1}_v(\tau^{n-1}_v)\bigr)\bigr)\,d\Qu[n-1]\Bigr)\,d\Ql[n-2]_\xi\\[1ex]
        &\hspace{-1.5em}= \int_{\Xil[n-2]} 1_{\gL_v}
            \Bigl(\int_{\Xi^0}g\bigl(s+u+\tau^0,
                    \gk\bigl(Z^0(\tau^0)\bigr)\bigr)\,dQ^0_v\Bigr)\Eval_{u=S_{n-1}}\,d\Ql[n-2]_\xi\\[1ex]
        &\hspace{-1.5em}=E_\xi\Bigl(E_v\bigr(g(s+u+S_1,K_1)\bigr)\eval_{u=S_{n-1}};\gL_v\Bigr)\\[1ex]
        &\hspace{-1.5em}=E_\xi\Bigl(E_{K_{n-1}}\bigr(g(s+u+S_1,K_1)\bigr)\eval_{u=S_{n-1}};\gL_v\Bigr)\\[1ex]
        &\hspace{-1.5em}=E_\xi\Bigl( (U_1 g)(s+S_{n-1},K_{n-1});\gL_v\Bigr),
\end{align*}
where in the last step we used definition~\eqref{eq2ii}. If in the preceding calculation
we replace the event $\{K_{n-1}=v\}$ by $\{K_{n-1}=\gD\}$, we get zero on both sides because
$K_{n-1}=\gD$ implies $K_n=\gD$ (see subsection~\ref{ssect2ii}). Thus summation over
$v\in V_c$ gives
\begin{equation*}
    E_\xi\bigl(g(s+S_n,K_n);\gL\bigr)
        = E_\xi\bigl((U_1 g)(s+S_{n-1}, K_{n-1});\gL\bigr),
\end{equation*}
and equation~\eqref{eq2iii} is proved for the case where $n\ge 2$ and $m=n-1$. As a
consequence we get
\begin{align*}
    (U_n g)(s,v)
        &= E_v\bigl(E_v\bigl(g(s+S_n,K_n)\cond \cB_{n-1}\bigr)\bigr)\\
        &= E_v\bigl((U_1 g)(s+S_{n-1},K_{n-1})\bigr)\\
        &= \bigl(U_{n-1}\comp U_1 g\bigr)(s,v).
\end{align*}
Now the general semigroup relation $U_{n+m}=U_n\comp U_m$, $m$, $n\in\N_0$, follows
by an application of Fubini's theorem.

Finally we show formula~\eqref{eq2iii} in the general case:
\begin{align*}
    E_\xi\bigl(g(s+S_n,K_n)&\cond \cB_m\bigr)\\
        &= E_\xi\bigl(E_\xi\bigl(g(s+S_n,K_n)\cond \cB_{n-1}\bigr)\cond \cB_m\bigr)\\
        &= E_\xi\bigl((U_1 g)(s+S_{n-1},K_{n-1})\cond\cB_m\bigr)\\
        &=\quad \dots \quad=\\
        &= E_\xi\bigl((U_1\comp\dotsb\comp U_1 g)(s+S_m,K_m)\cond\cB_m\bigr)\\
\end{align*}
where the multiple composition in the last expression involves $n-m$ operators $U_1$.
The semigroup property of $(U_n,\,n\in\N_0)$ implies formula~\eqref{eq2iii}.
\end{proof}

It will be useful to introduce some additional notation. For $r\in\N$, let
$\Rrp$ denote the set of all increasingly ordered $r$--tuples with entries in
$\R_+$. If $u\in\Rrp$ and $s\in\R$ we set $u+s=(u_1+s,\dotsc, u_r+s)\in\R^r$. $u<s$ means
that $u_i<s$ for all $i=1$, \dots, $r$ or equivalently $u_r<s$. The relations
$u>s$, $u\le s$, and $u\ge s$ are defined analogously. In particular,
when $s\le u$, then $u-s\in\Rrp$. For $r$, $q\in\N$, and $u\in\Rrp$, $w\in\Rrp[q]$
with $u\le w_1$, define
\begin{equation*}
    (u,w) = (u_1,\dotsc, u_r,w_1,\dotsc,w_q)\in\Rrp[r+q].
\end{equation*}
Furthermore, $Y(u)$ stands for $(Y(u_1),\dotsc, Y(u_r))$, and similarly for $Z^n_v(u)$,
$n\in\N_0$, $v\in V_c$.

In the sequel we shall consider random variables $W_m(h,g,u)$ of the following form
\begin{equation}    \label{eq2iv}
    W_m(h,g,u) = h\bigl(Y(u)\bigr)\,\chi_m(u)\,g(S_{m+1},K_{m+1}),
\end{equation}
where $m\in\N$, $h$ belongs to $\bcBG$,  $r\in \N$, $g$ to $\bRV$, and $u\in\Rrp$. Here
we have set
\begin{equation*}
    \chi_m(u) = 1_{\{S_m\le u <S_{m+1}\}}.
\end{equation*}
For $s\ge 0$ with $s\le u$ define
\begin{equation}    \label{eq2v}
    W_{m,s}(h,g,u)
        = h\bigl(Y(u-s)\bigr)\,\chi_m(u-s)\,g(s+S_{m+1},K_{m+1}),
\end{equation}
so that $W_{m,s=0}(h,g,u)=W_m(h,g,u)$. Moreover set
\begin{equation}    \label{eq2vi}
    R_m(h,g,u)(s,v) = E_v\bigl(W_{m,s}(h,g,u)\bigr),\qquad s\in\R_+,\,s\le u,\,v\in V_c.
\end{equation}
For the following it will be convenient to let $W_{m,s}(h,g,u)$ and $R_m(h,g,u)(s,v)$,
$v\in V_c$, be defined for all $s\in\R_+$. To this end we make the convention that
$Y(t)=\gD$ for all $t<0$. Then by $W_{m,s}(h,1,u)=W_m(h,1,u-s)$ the following formula
\begin{equation}    \label{eq2via}
    R_m(h,1,u)(s+t,v) = R_m(h,1,u-s)(t,v)
\end{equation}
holds for all $m\in\N$, $h\in B(\cG^r)$, $r\in\N$, $u\in\Rrp$, $s$, $t\in\R_+$,
$v\in V_c$.

Suppose that $r$, $q\in\N$, and that $h\in\bcBG$, $f\in\bcBG[q]$. Then $h\otimes f$
denotes the function in $B(\cG^{r+q})$ given by
\begin{equation}    \label{eq2vib}
\begin{split}
    h\otimes f&(\eta_1,\dotsc,\eta_{r+q})\\
        &= h(\eta_1,\dotsc,\eta_r)\,f(\eta_{r+1},\dotsc,\eta_{r+q}),\qquad
                    (\eta_1,\dotsc,\eta_{r+q})\in\cG^{r+q}.
\end{split}
\end{equation}

\begin{lemma}   \label{lem2iii}
Suppose that $r\in\N$, $u\in\Rrp$, and that $h\in\bcBG$.
\begin{enum_a}
    \item If $q\in\N$, $w\in \Rrp[q]$ with $u_r\le w$, and
            $f\in\bcBG[q]$, then
            \begin{subequations}     \label{eq2vii}
            \begin{equation}    \label{eq2viia}
                R_0\bigl(h\otimes f, 1, (u,w)\bigr)= R_0\big(M(f,w,u_r)h, 1,u\bigr)
            \end{equation}
            holds true, where $M(f,w,s)h\in B(\cG^r)$ is given by
            \begin{equation}    \label{eq2viib}
                \bigl(M(f,w,s)h\bigr)(\eta) = h(\eta)\,E_{\eta_r}\bigl(W_{0,s}(f,1,w)\bigr),
                                \quad \eta\in\cG^r,\,0\le s\le w.
            \end{equation}
            \end{subequations}
    \item If $g\in\bRV$, then
            \begin{subequations}     \label{eq2viii}
            \begin{equation}    \label{eq2viiia}
                R_0\bigl(h,g,u\bigr) = R_0\bigl(N(g,u_r)h,1,u\bigr)
            \end{equation}
            holds, where $N(g,s)h\in B(\cG^r)$ is given by
            \begin{equation}    \label{eq2viiib}
                \bigl(N(g,s) h\bigr)(\eta)= h(\eta)\,E_{\eta_r}\bigl(g(s+S_1,K_1)\bigr),
                                            \qquad \eta\in\cG^r,\,s\ge 0.
            \end{equation}
            \end{subequations}
\end{enum_a}
\end{lemma}

\begin{proof}
Both statements follow from the Markov property of the Brownian motion $Z^0$ on
$\cG_0$ underlying the construction of $Y$. We only prove statement~(b), the
proof of~(a) is similar and therefore omitted. Using the definition of $R_0$ and the
construction of $Y$, we compute for $s\in\R_+$, $v\in V_c$, as follows:
\begin{align*}
    R_0\bigl(h,g,u\bigr)(s,v)
        &= E_v\bigl(h\bigl(Y(u-s)\bigr)\,\chi_0(u-s)\,g(s+S_1,K_1)\bigr)\\
        &= E_v\bigl(h\bigl(Z^0(u-s)\bigr)\,1_{\{0\le u-s <\tau^0\}}\,
                g\bigl(s+\tau^0,\gk(Z^0(\tau^0))\bigr)\bigr).
\end{align*}
Recall that $\cF^0$ denotes the natural filtration of $Z^0$, and $\vt$ is a family
of shift operators for $Z^0$. It follows from the definition of the stopping time
$\tau^0$ and the path properties of $Z^0$, that on $\{\tau^0\ge u_r-s\}$ the
relation $\tau^0 = u_r-s + \tau^0\comp\vt_{u_r-s}$ holds true. Moreover, it is
easy to check that on this event we have $Z^0(\tau^0)=Z^0(\tau^0)\comp\vt_{u_r-s}$.
Therefore
\begin{align*}
    R_0\bigl(h,g,u\bigr)(s,v)
        &= E_v\Bigl(h\bigl(Z^0(u-s)\bigr)\,1_{\{0\le u-s <\tau^0\}}\\
        &\hspace{2em} \times E_v\bigl(g(u_r+\tau^0,\gk(Z^0(\tau^0)))
                            \comp\vt_{u_r-s}\cond \cF^0_{u_r-s}\bigr)\Bigr)\\
        &= E_v\Bigl(h\bigl(Z^0(u-s)\bigr)\,1_{\{0\le u-s <\tau^0\}}\\
        &\hspace{2em} \times E_{Z^0(u_r-s)}\bigl(g(u_r+\tau^0,\gk(Z^0(\tau^0)))\bigr)\Bigr)\\
        &= E_v\Bigl(h\bigl(Y(u-s)\bigr)\,\chi_0(u-s)\,
            E_{Y(u_r-s)}\bigl(g(u_r+S_1,K_1)\bigr)\Bigr)\\
        &= R_0\bigl(N(g,u_r)h,1,u\bigr)(s,v),
\end{align*}
and the proof is concluded.
\end{proof}

\begin{lemma}   \label{lem2iv}
For all $m$, $r\in\N$, $h\in\bcBG$, $g\in\bRV$, $u\in\Rrp$, $\xi\in\cG$, the formula
\begin{equation}    \label{eq2ix}
    E_\xi\bigl(W_m(h,g,u)\cond \cB_m\bigr)
        = R_0(h,g,u)(S_m,K_m)
\end{equation}
holds $Q_\xi$--a.s.
\end{lemma}

\begin{proof}
Observe that both side of equation~\eqref{eq2ix} vanish on the set $\{K_m=\gD\}$.
Let $\gL\in\cB_m$, $v\in V_c$, and set $\gL_v=\gL\cap\{K_m=v\}\in\cB_m$. Then
\begin{align*}
    E_\xi\bigl(&W_m(h,g,u);\gL_v\bigr)\\
        &= \int_{\Xil[m-1]} 1_{\gL_v}\, \Bigl(\int_{\Xiu[m]} h\bigl(Y(u)\bigr)\,
                    1_{\{S_m\le u<S_{m+1}\}}\\
        &\hspace{7em} \times g(S_{m+1},K_{m+1})\,d\Qu[m]\Bigr)\,d\Ql[m-1]_\xi\\
        &= \int_{\Xil[m-1]} 1_{\gL_v}\, \Bigl(\int_{\Xi^m} h\bigl(Z^m_v(u-s)\bigr)\,
                    1_{\{0\le u-s < \tau^m_v\}}\\
        &\hspace{7em} \times g\bigl(s+\tau^m_v,\gk(Z^m_v(\tau^m_v))\bigr)\,dQ^m\Bigr)
                \Eval_{s=S_m}\,d\Ql[m-1]_\xi\\
        &= \int_{\Xil[m-1]} 1_{\gL_v}\, \Bigl(\int_{\Xi^0} h\bigl(Z^0(u-s)\bigr)\,
                    1_{\{0\le u-s < \tau^0\}}\\
        &\hspace{7em} \times g\bigl(s+\tau^0,\gk(Z^0(\tau^0))\bigr)\,dQ^0_v\Bigr)
                \Eval_{s=S_m}\,d\Ql[m-1]_\xi\\
        &= E_\xi\Bigl(E_v\bigl(h\bigl(Y(u-s)\bigr)\,\chi_0(u-s)\,g(s+S_1,K_1)\bigr)
            \eval_{s=S_m};\,\gL_v\Bigr)\\
        &= E_\xi\Bigl(E_{K_m}\bigl(h\bigl(Y(u-s)\bigr)\,\chi_0(u-s)\,g(s+S_1,K_1)\bigr)
            \eval_{s=S_m};\,\gL_v\Bigr)\\
        &= E_\xi\bigl(R_0(h,g,u)(S_m,K_m);\,\gL_v\bigr).
\end{align*}
Summation over $v\in V_c$ finishes the proof.
\end{proof}

\begin{lemma}   \label{lem2v}
For all $m$, $r\in\N$, $u\in\Rrp$, $h\in\bcBG$, $g\in\bRV$,
\begin{equation}    \label{eq2x}
    R_m(h,g,u)(s,v) = \bigl(U_m R_0(h,g,u)\bigr)(s,v),\qquad s\in\R_+,\,v\in V_c,
\end{equation}
holds.
\end{lemma}

\begin{proof}
By definition of $U_m$
\begin{equation*}
    \bigl(U_m R_0(h,g,u)\bigr)(s,v)
        = E_v\bigl(R_0(h,g,u)(s+S_m,K_m)\bigr).
\end{equation*}
With formula~\eqref{eq2via} and lemma~\ref{lem2iv} we find
\begin{align*}
    \bigl(U_m R_0(h,g,u)\bigr)(s,v)
        &= E_v\bigl(R_0(h,g(s+\,\cdot\,),u-s)(S_m,K_m)\bigr)\\
        &= E_v\bigl(E_v\bigl(W_m(h,g(s+\,\cdot\,),u-s)\cond \cB_m\bigr)\bigr)\\
        &= E_v\bigl(W_m(h,g(s+\,\cdot\,),u-s)\bigr)\\
        &= E_v\bigl(W_{m,s}(h,g,u)\bigr)\\
        &= R_m(h,g,u)(s,v).\qedhere
\end{align*}
\end{proof}

\begin{corollary}   \label{cor2vi}
For all $m$, $n$, $r\in\N$, $u\in\Rrp$, $h\in\bcBG$, $g\in\bRV$,
\begin{equation}    \label{eq2xi}
    U_n R_m(h,g,u) = R_{n+m}(h,g,u).
\end{equation}
is valid.
\end{corollary}

\begin{proof} By lemma~\ref{lem2v} and lemma~\ref{lem2ii}, statement~(b), we obtain
\begin{equation*}
    U_n R_m(h,g,u) = U_n\comp U_m R_0(h,g,u) = U_{n+m} R_0(h,g,u) = R_{n+m}(h,g,u).\qedhere
\end{equation*}
\end{proof}

\begin{lemma}   \label{lem2vii}
For all $m$, $n\in\N$, $m\le n$, $r\in\N$, $u\in\Rrp$, $h\in\bcBG$, $g\in\bRV$, $\xi\in\cG$,
the following formula holds true:
\begin{equation}    \label{eq2xii}
    E_\xi\bigl(W_n(h,g,u)\cond \cB_m\bigr) = R_{n-m}(h,g,u)(S_m,K_m).
\end{equation}
\end{lemma}

\begin{proof}
Apply lemma~\ref{lem2iv} to compute as follows
\begin{align*}
    E_\xi\bigl(W_n(h,g,u)\cond \cB_m\bigr)
        &= E_\xi\bigl(E_\xi\bigl(W_n(h,g,u)\cond \cB_n\bigr)\cond \cB_m\bigr)\\
        &= E_\xi\bigl(R_0(h,g,u)(S_n,K_n)\cond \cB_m\bigr)\\
        &= \bigl(U_{n-m} R_0(h,g,u)\bigr)(S_m,K_m),
\end{align*}
where we used lemma~\ref{lem2ii}, formula~\eqref{eq2iii}, in the last step. An application of
lemma~\ref{lem2v} concludes the proof.
\end{proof}

With these preparations, we are ready for the

\begin{proof}[Proof of Proposition~\ref{prop2i}]
Assume that $f\in\bcB(\cG)$, $s$, $t\in\R_+$, and that $\xi\in\cG$. Since
$(S_m,\,m\in\N)$ $Q_\xi$--a.s.\ strictly increases to $+\infty$, and since $S_m$,
$m\in\N$, is an $\cF^Y$--stopping time (cf., lemma~\ref{lemA} in appendix~\ref{appA}),
it suffices to prove that equation~\eqref{eq2i} holds $Q_\xi$--a.s.\ for every
$m\in\N_0$ on $\{S_m\le s <S_{m+1}, Y(s)\ne \gD\}\in\cF^Y_s$. We fix an arbitrary
$m\in\N_0$. Clearly, the family of random variables of the form
\begin{equation*}
	g\bigl(Y(w)\bigr)\,1_{\{0\le w <S_m\}}\,W_m(h,1,u)
\end{equation*}
with $r$, $q\in\N$, $u\in\Rrp$, $u_r=s$, $w\in\Rrp[q]$, $h\in\bcBG$, and
$g\in\bcBG[q]$, generates the $\gs$--algebra $\cF_s^Y\cap\{S_m\le s<S_{m+1}\}$.
Therefore it is sufficient to show that
\begin{equation}   \label{eq2xiii}
\begin{split}
    E_\xi\bigl(g\bigl(Y(w)\bigr)\,&1_{\{0\le w <S_m\}}\,
				W_m(h,1,u)\,f\bigl(Y(s+t)\bigr)\bigr)\\[1ex]
        &= E_\xi\bigl(g\bigl(Y(w)\bigr)\,1_{\{0\le w <S_m\}}\,W_m(h,1,u)\,
            E_{Y(s)}\bigl(f(Y(t))\bigr)\bigr),
\end{split}
\end{equation}
holds for all $r$, $q\in\N$, $u\in\Rrp$ with $u_r=s$, $w\in\Rrp[q]$, $h\in\bcBG$,
$g\in\bcBG[q]$ and $f\in B(\cG)$. (Since the random variables under the
expectation signs of both sides of equation~\eqref{eq2xiii} vanish on the
set $\{Y(s)=\gD\}$, we can henceforth safely ignore the condition $Y(s)\ne \gD$.)
Expand the left hand side of equation~\eqref{eq2xiii} as follows:
\begin{equation}   \label{eq2xiv}
\begin{split}
    E_\xi\bigl(g\bigl(&Y(w)\bigr)\,1_{\{0\le w <S_m\}}\,W_m(h,1,u)\,f\bigl(Y(s+t)\bigr)\bigr)\\[1ex]
        &= \sum_{n=m}^\infty E_\xi\bigl(g\bigl(Y(w)\bigr)\,1_{\{0\le w <S_m\}}\,W_m(h,1,u)\,
                W_n(f,1,s+t)\bigr).
\end{split}
\end{equation}
Consider the summand with $n=m$, which is of the form
\begin{align*}
    E_\xi\bigl(&g\bigl(Y(w)\bigr)\,1_{\{0\le w < S_m\}}\,W_m(h\otimes f,1,(u,s+t))\bigr)\\
        &= E_\xi\bigl(g\bigl(Y(w)\bigr)\,1_{\{0\le w < S_m\}}\,
            E_\xi\bigl(W_m(h\otimes f,1,(u,s+t))\cond \cB_m\bigr)\bigr)\\
        &= E_\xi\bigl(g\bigl(Y(w)\bigr)\,1_{\{0\le w < S_m\}}
                    \,R_0(h\otimes f, 1, (u,s+t))(S_m,K_m)\bigr),
\end{align*}
where we made use of formula~\eqref{eq2ix}. Now we apply statement~(a) of lemma~\ref{lem2iii}
with the choice $q=1$ which yields (recall that $u_r=s$)
\begin{align*}
    E_\xi\bigl(&g\bigl(Y(w)\bigr)\,1_{\{0\le w < S_m\}}\,W_m(h\otimes f,1,(u,s+t))\bigr)\\
        &= E_\xi\bigl(g\bigl(Y(w)\bigr)\,1_{\{0\le w < S_m\}}
                    \,R_0(M(f,s+t,s)h, 1, u)(S_m,K_m)\bigr)\\
        &= E_\xi\bigl(g\bigl(Y(w)\bigr)\,1_{\{0\le w < S_m\}}
                    \,E_\xi\bigl(W_m(M(f,s+t,s)h, 1, u)\cond\cB_m\bigr)\bigr)\\
        &= E_\xi\bigl(g\bigl(Y(w)\bigr)\,1_{\{0\le w < S_m\}}\,
                    W_m(M(f,s+t,s)h, 1, u)\bigr),
\end{align*}
where we used formula~\eqref{eq2ix} again. Combining~\eqref{eq2iv} with~\eqref{eq2viib}
in $W_m(M(f,s+t,s)h, 1, u)$ we thus have shown
\begin{equation}    \label{eq2xv}
\begin{split}
    E_\xi\bigl(&g\bigl(Y(w)\bigr)\,1_{\{0\le w < S_m\}}\,W_m(h,1,u)\,W_m(f,1,s+t)\bigr)\\
        &= E_\xi\bigl(g\bigl(Y(w)\bigr)\,1_{\{0\le w < S_m\}}\,W_m(h,1,u)\,
                E_{Y(s)}\bigl(f(Y(t))\,1_{\{0<t<S_1\}}\bigr)\bigr).
\end{split}
\end{equation}
Next consider a generic summand with $n>m$ on the right hand side of~\eqref{eq2xiv}. Then
\begin{align*}
    E_\xi\bigl(&g\bigl(Y(w)\bigr)\,1_{\{0\le w<S_m\}}\,W_m(h,1,u)\,W_n(f,1,s+t)\bigr)\\
        &= E_\xi\bigl(g\bigl(Y(w)\bigr)\,1_{\{0\le w<S_m\}}\,W_m(h,1,u)\,
                E_\xi\bigl(W_n(f,1,s+t)\cond \cB_{m+1}\bigr)\bigr)\\
        &= E_\xi\bigl(g\bigl(Y(w)\bigr)\,1_{\{0\le w<S_m\}}\,W_m(h,1,u)\,
                    R_{n-m-1}(f,1,s+t)(S_{m+1},K_{m+1})\bigr)\\
        &= E_\xi\bigl(g\bigl(Y(w)\bigr)\,1_{\{0\le w<S_m\}}\,W_m(h,R_{n-m-1}(f,1,s+t),u)\bigr)
\end{align*}
where we used lemma~\ref{lem2vii} in the second step. Conditioning on $\cB_m$ gives
\begin{align*}
    E_\xi\bigl(&g\bigl(Y(w)\bigr)\,1_{\{0\le w<S_m\}}\,W_m(h,1,u)\,W_n(f,1,s+t)\bigr)\\
        &= E_\xi\bigl(g\bigl(Y(w)\bigr)\,1_{\{0\le w<S_m\}}\,
                E_\xi\bigl(W_m(h,R_{n-m-1}(f,1,s+t),u)\cond \cB_m\bigr)\bigr)\\
        &= E_\xi\bigl(g\bigl(Y(w)\bigr)\,1_{\{0\le w<S_m\}}\,
                R_0(h,R_{n-m-1}(f,1,s+t),u)(S_m,K_m)\bigr)\\
        &= E_\xi\bigl(g\bigl(Y(w)\bigr)\,1_{\{0\le w<S_m\}}\,
                R_0\bigl(N(R_{n-m-1}(f,1,s+t),s)h,1,u\bigr)(S_m,K_m)\bigr).
\end{align*}
Here we used lemmas~\ref{lem2iv}, \ref{lem2iii}.b, and in the last step also $u_r=s$.
Applying lemma~\ref{lem2iv}, we get
\begin{align*}
    E_\xi\bigl(g\bigl(Y(w)\bigr)\,&1_{\{0\le w<S_m\}}\,W_m(h,1,u)\,W_n(f,1,s+t)\bigr)\\
	   &=  E_\xi\Bigl(g\bigl(Y(w)\bigr)\,1_{\{0\le w<S_m\}}\\
	   &\hspace{4em}\times E_{\xi}\bigl(W_m(N(R_{n-m-1}(f,1,s+t),s)h,1,u)
				\cond \cB_m\bigr)\Bigr)\\
        &= E_\xi\bigl(g\bigl(Y(w)\bigr)\,1_{\{0\le w<S_m\}}\,
                W_m\bigl(N(R_{n-m-1}(f,1,s+t),s)h,1,u\bigr)\bigr)\\
        &= E_\xi\bigl(g\bigl(Y(w)\bigr)\,1_{\{0\le w<S_m\}}\,
                h\bigl(Y(u)\bigr)\,\chi_m(u)\\
        &\hspace{4em}\times E_{Y(s)}\bigl(R_{n-m-1}(f,1,s+t)(s+S_1,K_1)\bigr)\bigr).
\end{align*}
For $\eta\in\cG$ relation~\eqref{eq2via} yields
\begin{align*}
    E_\eta\bigl(R_{n-m-1}(f,1,s+t)(s+S_1,K_1)\bigr)
        &= E_\eta\bigl(R_{n-m-1}(f,1,t)(S_1,K_1)\bigr)\\
        &= E_\eta\bigl(E_\eta\bigl(W_{n-m}(f,1,t)\cond\cB_1\bigr)\bigr)\\
        &= E_\eta\bigl(W_{n-m}(f,1,t)\bigr),
\end{align*}
with another application of formula~\eqref{eq2xii}. With the choice $\eta=Y(s)$ this
relation therefore gives
\begin{equation}    \label{eq2xvi}
\begin{split}
    E_\xi\bigl(g\bigl(Y(w)\bigr)\,&1_{\{0\le w<S_m\}}\,W_m(h,1,u)\,W_n(f,1,s+t)\bigr)\\
        &= E_\xi\Bigl(g\bigl(Y(w)\bigr)\,1_{\{0\le w<S_m\}}\,h\bigl(Y(u)\bigr)\,\chi_m(u)\\
        &\hspace{7em}\times E_{Y(s)}\bigl(f(Y(t))\,1_{\{S_{n-m}\le t<S_{n-m+1}\}}\bigr)\Bigr).
\end{split}
\end{equation}
Formulae~\eqref{eq2xv} and~\eqref{eq2xvi} entail
\begin{align*}
    \sum_{n=m}^\infty &E_\xi\Bigl(g\bigl(Y(w)\bigr)\,1_{\{0\le w <S_m\}}\,W_m(h,1,u)\,
                W_n(f,1,s+t)\Bigr)\\[1ex]
        &= E_\xi\Bigl(g\bigl(Y(w)\bigr)\,1_{\{0\le w < S_m\}}\,W_m(h,1,u)\,
                E_{Y(s)}\bigl(f(Y(t))\,1_{\{0\le t<S_1\}}\bigr)\Bigr)\\
        &\hspace{4em}+ \sum_{n=m+1}^\infty E_\xi\Bigl(g\bigl(Y(w)\bigr)\,1_{\{0\le w<S_m\}}\,
                h\bigl(Y(u)\bigr)\,\chi_m(u)\\
        &\hspace{12em}\times E_{Y(s)}\bigl(f(Y(t))\,1_{\{S_{n-m}\le t<S_{n-m+1}\}}\bigr)\Bigr)\\[1ex]
        &= E_\xi\Bigl(g\bigl(Y(w)\bigr)\,1_{\{0\le w < S_m\}}\,W_m(h,1,u)\,
                E_{Y(s)}\bigl(f(Y(t))\bigr)\Bigr),
\end{align*}
which proves equation~\eqref{eq2xiii}.
\end{proof}

\subsection{A Brownian motion on \boldmath $\cG$ and its generator}\label{ssect2iv}

The stochastic process $Y$ and its underlying probability family $(\Xi,\cC,Q)$, $Q =
(Q_\xi,\,\xi\in\cG)$, are not very convenient to work with. Therefore we introduce
another version in this subsection. As the underlying sample space $\gO$ we choose
the path space $\CD(\R_+,\cG)$ of $Y$ endowed with the $\gs$--algebra $\cA$ generated
by the cylinder sets of $\CD(\R_+,\cG)$. Obviously, $Y$ is a measurable mapping from
$(\Xi,\cC)$ into $(\gO,\cA)$. For $\xi\in\cG$ let $P_\xi$ denote the image measure
of $Q_\xi$ under $Y$. Set $P=(P_\xi,\,\xi\in\cG)$. Moreover, let the canonical
coordinate process on $(\gO,\cA)$ be denoted by  $X=(X_t,\,t\in\R_+)$. Clearly, $X$
is a version of $Y$. We set $X_{+\infty}=\gD$, and denote the natural filtration of
$X$ by $\cF=(\cF_t,\,t\in\R_+)$. As usual $\cF_\infty$ stands for
$\gs(\cF_t,\,t\in\R_+)$. Whenever it is notationally more convenient we shall also
write $X(t)$ for $X_t$, $t\in\R_+$.

Let $H_V$ denote the hitting time of the set $V$ of vertices of $\cG$ by $X$:
\begin{equation*}
    H_V = \inf\{t > 0,\,X_t\in V\}.
\end{equation*}
Suppose that $X$ starts in $\xi\in l^\circ$, $l\in\cI\cup\cE$, and that $l$ is
isomorphic to the interval $I$. Then it follows directly from the discussion at the
end of subsection~\ref{ssect2ii} that the stopped process $X(\,\cdot\,\land H_V)$ is
equivalent to a standard Brownian motion on $I$ with absorption in the endpoint(s)
of $I$. The necessary path properties of $X$ being obvious, we therefore find that
$X$ satisfies all defining properties of a Brownian motion on $\cG$ (cf.\
section~\ref{sect1}, and definition~\Iref{def_3_1}), except that we still have to
prove its strong Markov property. This will be done next.

Let $\theta = (\theta_t,\,t\in\R_+)$ denote the natural family of shift operators on
$\gO$: $\theta_t(\go) = \go(t+\,\cdot\,)$ for $\go\in\gO$. Thus in particular
$\theta$ is a family of shift operators for $X$. Since the simple Markov property is
a property of the finite dimensional distributions of a stochastic process, and the
finite dimensional distributions of $X$ and $Y$ coincide, it immediately follows
from proposition~\ref{prop2i} that $X$ is a Markov process. Then standard monotone
class arguments (e.g., \cite{KaSh91, ReYo91}) give the Markov property in the familiar
general form:

\begin{proposition}   \label{prop2viii}
Assume that $\xi\in\cG$, $t\in\R_+$, and that $W$ is an $\cF_\infty$--measurable,
positive or integrable random variable on $(\gO,\cA, P)$. Then
\begin{equation}    \label{eq2xvii}
    E_\xi\bigl(W\comp \theta_t \cond \cF_t\bigr) = E_{X_t}\bigl(W\bigr),
\end{equation}
holds true $P_\xi$--a.s.\ on $\{X_t\ne \gD\}$.
\end{proposition}

A routine argument based on the path properties of $X$ (similar to, but much easier
than the one used in the proof of lemma~\ref{lemA} in appendix~\ref{appA}) shows
that $H_V$ is an $\cF$--stopping time. We have the following

\begin{lemma}   \label{lem2ix}
$X$ has the strong Markov property with respect to the hitting time $H_V$. That is,
for all $\xi\in\cG$, $t\in\R_+$, $f\in\bcB(\cG)$,
\begin{equation}    \label{eq2xviii}
    E_\xi\bigl(f(X_{t+H_V})\cond \cF_{H_V}\bigr)
        = E_{X_{H_V}}\bigl(f(X_t)\bigr)
\end{equation}
holds true $P_\xi$--a.s.
\end{lemma}

\begin{proof}
To begin with, observe that since $\gO=\CD(\cG)$ and $X$ is the canonical coordinate
process, there is a natural family $\ga=(\ga_t,\,t\in\R_+)$ of stopping operators for
$X$, namely $\ga_t(\go) = \go(\,\cdot\,\land t)$, $t\in\R_+$. Therefore we get that
$\cF_T = \gs(X_{s\land T},\,s\in\R_+)$ for any stopping time $T$ relative to $\cF$.
Indeed, one can show this along the same lines used to prove Galmarino's
theorem (e.g., \cite[p.~458]{Ba91}, \cite[p.~86]{ItMc74}, \cite[p.~43~ff]{Kn81},
\cite[p.~45]{ReYo91}). Therefore it is sufficient to prove that for all $n\in\N$,
$s_1$, \dots, $s_n\in\R_+$, $t\in\R_+$, $\xi\in\cG$, and all $g\in\bcB(\cG^n)$,
$f\in\bcB(\cG)$, the following formula
\begin{equation}    \label{eq2xix}
\begin{split}
    E_\xi\bigl(g(X(s_1\land H_V),&\dotsc,X(s_1\land H_V))\,f(X(t+ H_V))\bigr)\\
        &= E_\xi\bigl(g(X(s_1\land H_V),\dotsc,X(s_1\land H_V))\,
                     E_{X(H_V)}\bigl(f(X(t))\bigr)\bigr)
\end{split}
\end{equation}
holds. Recall that $S_V$ denotes the hitting time of $V$ by $Y$. Since $P_\xi$ is
the image of $Q_\xi$ under $Y$, and since $S_V = H_V\comp Y$,
equation~\eqref{eq2xix} is equivalent to
\begin{equation}    \label{eq2xx}
    E_\xi\bigl(G\,f(Y(t+S_V))\bigr)
        = E_\xi\bigl(G\,E_{Y(S_V)}\bigl(f(Y(t))\bigr)\bigr),
\end{equation}
where we have set
\begin{equation*}
    G = g\bigl(Y(s_1\land S_V),\dotsc, Y(s_n\land S_V)\bigr).
\end{equation*}
Recall that $S_1$ denotes the hitting time of $V_c\subset V$ by $Y$, so that $S_V\le
S_1$. $Y$ is progressively measurable relative to $\cF^Y$ which entails that $G$ is
measurable with respect to $\cF^Y_{S_V}\subset \cF^Y_{S_1}\subset \cB_1$ (see also
the corresponding argument in the proof of lemma~\ref{lemA}). Using the notation of
subsection~\ref{ssect2iii} we write
\begin{equation}    \label{eq2xxi}
\begin{split}
    E_\xi\big(G\, &f(Y(t+S_V))\bigr)\\
        &= E_\xi\big(G\, f(Y(t+S_V);\,S_V\le t+S_V < S_1)\bigr)\\
        &\hspace{4em} + \sum_{n=1}^\infty E_\xi\bigl(G\,E_\xi\bigl(f(Y(t+u))\,
                    \chi_n(t+u)\cond \cB_1\bigr)\eval_{u=S_V}\bigr).
\end{split}
\end{equation}
For the last equality --- similarly as in the proof of lemma~\ref{lem2ii} ---
we made use of the product structure of the probability space $(\Xi,\cC,Q_\xi)$,
and the fact that $S_V\le S_1$, which entails that $S_V$ only depends on the
variable $\go^0\in \Xi^0$. By lemma~\ref{lem2vii} and formula~\eqref{eq2via} we
get for $u\le S_1$, $n\in\N$,
\begin{align*}
    E_\xi\bigl(f(Y(t+u))\,\chi_n(t+u)\cond \cB_1\bigr)
        &= R_{n-1}(f,1,t+u)(S_1,K_1)\\
        &= R_{n-1}(f,1,t)(S_1-u,K_1).
\end{align*}
Then
\begin{align*}
    E_\xi\bigl(G\, R_{n-1}&(f,1,t)(S_1-S_V,K_1)\bigr)\\
        &= E_\xi\bigl(G\,R_{n-1}(f,1,t)(0,K_1);\,S_V=S_1\bigr)\\
        &\hspace{4em} +E_\xi\bigl(G\,R_{n-1}(f,1,t)(S_1-S_V,K_1);\,S_V<S_1\bigr)\\
        &= E_\xi\bigl(G\,E_{Y(S_V)}\bigl(f(Y(t))\,\chi_{n-1}(t)\bigr);\,S_V=S_1\bigr)\\
        &\hspace{4em} +E_\xi\bigl(G\,R_{n-1}(f,1,t)(S_1-S_V,K_1);\,S_V<S_1\bigr),
\end{align*}
because on $S_V=S_1$, $Y(S_V)=Y(S_1)=K_1$. The second term on the right hand side of
the last equality only involves the random variables $Y(s_i\land S_V)$, $S_1$,
$S_V$, and $K_1$. They are all defined in terms of the strong Markov process $Z^0$
underlying the construction of $Y$ (cf.\ section~\ref{ssect2ii}). Moreover, on the
event $\{S_V < S_1\}$ we get from the definition of $S_1$ as the hitting time of
$V_c$ that $S_1 = S_V + S_1\comp \vt_{S_V}$. Also, on $\{S_V < S_1\}$,
$K_1 = K_1\circ\vt_{S_V}$ holds true.
On the other hand $G$ is measurable with respect to $\cF^0_{S_V}$, where $\cF^0$ is
the natural filtration of $Z^0$. Thus the strong Markov property of $Z^0$ gives
\begin{align*}
    E_\xi\bigl(G\,&R_{n-1}(f,1,t)(S_1-S_V,K_1);\,S_V<S_1\bigr)\\
        &= E_\xi\bigl(G\,E_{Y(S_V)}\bigl(R_{n-1}(f,1,t)(S_1,K_1)\bigr);\,S_V<S_1\bigr)\\
        &= E_\xi\bigl(G\,E_{Y(S_V)}\bigl(f(Y(t))\,\chi_n(t)\bigr);\,S_V<S_1\bigr).
\end{align*}
In the last step we used for $\eta\in\cG$,
\begin{align*}
    E_\eta\bigl(R_{n-1}(f,1,t)(S_1,K_1)\bigr)
        &= E_\eta\bigl(E_\eta\bigl(W_n(f,1,t)\cond \cB_1\bigr)\bigr)\\
        &= E_\eta\bigl(f(Y(t))\,\chi_n(t)\bigr),
\end{align*}
with another application of lemma~\ref{lem2vii}, and then we made the choice $\eta=
Y(S_V)$. Similarly, for the first term on the right hand side of equation~\eqref{eq2xxi}
we can use the strong Markov property of $Z^0$ (together with $\{t+S_V<S_1\}
= \{t<S_1\comp\vt_{S_V}\}\cap\{S_V<S_1\}$) to show that
\begin{equation*}
\begin{split}
    E_\xi\bigl(G\,f(Y(t+S_V));\,&S_V\le t+S_V < S_1\bigr)\\
        &= E_\xi\bigl(G\,E_{Y(S_V)}\bigl(f(Y(t);\,t<S_1)\bigr);\,S_V<S_1\bigr).
\end{split}
\end{equation*}
Inserting these results into the right hand side of formula~\eqref{eq2xxi}, we find
\begin{align*}
    E_\xi\bigl(&G\,f(Y(t+S_V))\bigr)\\
        &= E_\xi\bigl(G\,E_{Y(S_V)}\bigl(f(Y(t));\,t<S_1\bigr);\,S_V<S_1\bigr)\\
        &\hspace{4em} +\sum_{n=1}^\infty E_\xi\bigl(G\,E_{Y(S_V)}\bigl(f(Y(t))\,
                            \chi_n(t)\bigr);\,S_V<S_1\bigr)\\
        &\hspace{4em} +\sum_{n=0}^\infty E_\xi\bigl(G\,E_{Y(S_V)}\bigl(f(Y(t))\,
                            \chi_n(t)\bigr);\,S_V=S_1\bigr)\\
        &= E_\xi\bigl(G\,E_{Y(S_V)}\bigl(f(Y(t))\bigr)\bigr).
\end{align*}
Thus equation~\eqref{eq2xx} holds and lemma~\ref{lem2ix} is proved.
\end{proof}

\begin{proposition}   \label{prop2x}
$X$ is a Feller process.
\end{proposition}

\begin{proof}
It is well-known, that it is sufficient to prove~(i) that the resolvent of $X$
preserves $C_0(\cG)$, and~(ii) that for all $f\in C_0(\cG)$, $\xi\in\cG$,
$E_\xi\bigl(f(X_t)\bigr)$ converges to $f(\xi)$ as $t$ decreases to $0$. (A complete
proof can be found in appendix~\ref{I_app_A} of~\cite{BMMG1}.) Statement~(ii) immediately
follows by an application of the dominated convergence theorem and the fact that $X$
is a normal process with right continuous paths.

To prove statement~(i) consider the resolvent $R = (R_\gl,\,\gl>0)$ of $X$, and let
$\gl>0$. Since $X$ is strongly Markovian with respect to the hitting time $H_V$ of
the set of vertices $V$ (lemma~\ref{lem2ix}), we get for $\xi\in\cG$,
$f\in\bcB(\cG)$ the first passage time formula
\begin{equation}    \label{eq2xxii}
    (R_\gl f)(\xi)
        = (R^D_\gl f)(\xi) + E_\xi\bigl(e^{-\gl H_V}\,(R_\gl f)(X_{H_V})\bigr),
\end{equation}
where $R^D$ is the Dirichlet resolvent. That is, $R^D$ is the resolvent of the
process $X$ with killing at the moment of reaching a vertex of $\cG$. Recall the
equivalence of the stopped process $X(\,\cdot\,\land H_V)$ with the Brownian motion
with absorption on the corresponding interval $I$ stated at the beginning of this
subsection. Then we can give explicit expressions for all entities appearing in the
first passage time formula~\eqref{eq2xxii}. Using the well-known formulae for
Brownian motions on the real line (see, e.g., \cite[p.~73ff]{DyJu69},
\cite[p.~29f]{ItMc74}) we find for $\xi$, $\eta\in\cG^\circ$,
\begin{subequations}    \label{eq2xxiii}
\begin{equation}    \label{eq2xxiiia}
    r^D_\gl(\xi,\eta)
        = \sum_{i\in \cI} r^D_{\gl,i}(\xi,\eta)\,1_{\{\xi,\eta\in i\}}
            + \sum_{e\in\cE} r^D_{\gl,e}(\xi,\eta)\,1_{\{\xi,\eta\in e\}},
\end{equation}
with
\begin{equation}    \label{eq2xxiiib}
    r^D_{\gl,i}(\xi,\eta)
        = \frac{1}{\sgl}\,\sum_{k\in\Z}\Bigl(e^{-\sgl|x-y+2ka_i|}-e^{-\sgl|x+y+2ka_i|}\Bigr),
\end{equation}
where in local coordinates $\xi=(i,x)$, $\eta=(i,y)$, $x$, $y\in (0,a_i)$. In the case of
an external edge $e$, we get
\begin{equation}    \label{eq2xxiiic}
    r^D_{\gl,e}(\xi,\eta) = \frac{1}{\sgl}\,\Bigl(e^{-\sgl|x-y|}-e^{-\sgl(x+y)}\Bigr),
\end{equation}
\end{subequations}
with $\xi=(e,x)$, $\eta=(e,y)$, $x$, $y\in (0+\infty)$. Remark that both kernels
vanish whenever $\xi$ or $\eta$ converge from the interior of any edge to a vertex
to which the edge is incident.

Consider the second term on the right hand side of equation~\eqref{eq2xxii}. Suppose that
$\xi\in i^\circ$, $i\in \cI$, and that $i$ is isomorphic to $[0,a_i]$. Assume furthermore
that $v_1$, $v_2$ are the vertices in $V$ to which $i$ is incident, and that under this
isomorphism $v_1$ corresponds to $0$, while $v_2$ corresponds to $a_i$. Then we get
\begin{equation*}
\begin{split}
    E_\xi\bigl(e^{-\gl H_V}\,(R_\gl f)(X_{H_V})\bigr)
        = E_\xi\bigl(&e^{-\gl H_{v_1}};\,H_{v_1}<H_{v_2}\bigr)\,(R_\gl f)(v_1)\\
            &+  E_\xi\bigl(e^{-\gl H_{v_2}};\,H_{v_2}<H_{v_1}\bigr)\,(R_\gl f)(v_2),
\end{split}
\end{equation*}
because $X$ has paths which are continuous up to the lifetime of $X$, and $X$ cannot
be killed before reaching a vertex. Here $H_{v_k}$, $k=1$, $2$, denotes the hitting
time of the vertex $v_k$.  The expectation values in the last line are those of a
standard Brownian motion and they are well-known, too (see, e.g.,
\cite[p.~73ff]{DyJu69}, \cite[p.~29f]{ItMc74}). Thus for $\xi = (i,x)$, $x\in
[0,a_i]$,
\begin{equation}    \label{eq2xxiv}
\begin{split}
    E_\xi\bigl(&e^{-\gl H_V}\,(R_\gl f)(X_{H_V})\bigr)\\
        &= \frac{\sinh\bigl(\sgl (a_i-x)\bigr)}{\sinh\bigl(\sgl a_i\bigr)}\,(R_\gl f)(v_1)
            + \frac{\sinh\bigl(\sgl x\bigr)}{\sinh\bigl(\sgl a_i\bigr)}\,(R_\gl f)(v_2).
\end{split}
\end{equation}
Similarly, for $\xi\in e^\circ$ with local coordinates $\xi=(e,x)$, $x\in (0,+\infty)$
we find
\begin{equation}    \label{eq2xxv}
    E_\xi\bigl(e^{-\gl H_V}\,(R_\gl f)(X_{H_V})\bigr)
        = e^{-\sgl x}\,(R_\gl f)(v),
\end{equation}
where $v$ is the vertex to which $e$ is incident.

With the formulae~\eqref{eq2xxii}--\eqref{eq2xxv} it is straightforward to check
that $R_\gl$ maps $C_0(\cG)$ into itself, and the proof is complete.
\end{proof}

Since $X$ has right continuous paths, standard results (see, e.g,
\cite[Theorem~III.3.1]{ReYo91}, or \cite[Theorem~III.15.3]{Wi79}) provide the

\begin{corollary}   \label{cor2xi}
$X$ is strongly Markovian.
\end{corollary}

Thus we have also proved the

\begin{corollary}   \label{cor2xii}
$X$ is a Brownian motion on $\cG$ in the sense of definition~\Iref{def_3_1}.
\end{corollary}

It remains to calculate the domain of the generator of $X$, i.e., the boundary
conditions at the vertices. Let $v\in V$, and assume that $X$ starts in $v$. Then by
construction of $X$ and $Y$, $X$ is equivalent to $Z^0$ with start in $v$ up to its
first hitting of a shadow vertex. That is, $X$ is equivalent to $Z^0$ up to the
first time $X$ hits a vertex different from $v$. It follows that if $v$ is absorbing
for $Z^0$, then it is so for $X$, and if $v$ is an exponential holding point with
jump to $\gD$, then it is also so for $X$ with the same exponential rate. In
particular, $v$ is a trap for $X$ if and only if it is a trap for $Z^0$. If $v$ is
not a trap, then we can use Dynkin's formula~\cite[p.~140, ff.]{Dy65a} to calculate
the boundary condition implemented by $X$. Clearly, this gives the same boundary
conditions as for $Z^0$, because Dynkin's formula only involves an arbitrary small
neighborhood of the vertex. (See also the corresponding arguments in articles~I
and~II.) Thus we have proved the following

\begin{theorem} \label{thm2xiii}
$X$ is a Brownian motion on $\cG$ whose generator has a domain characterized by
the same boundary conditions as the generator of $Z^0$.
\end{theorem}

\section{Proof of Theorem~\ref{thm1ii}}  \label{sect3}

Suppose that $\cG=(V,\cI,\cE,\p)$ is a metric graph without tadpoles. Let data $a$,
$b$, $c$ as in~\eqref{eq1i} be given which satisfy equation~\eqref{eq1ii}.

With every $v\in V$ we associate a single vertex graph $\cG(v)$ consisting of the
vertex $v$ and $|\cL(v)|$ external edges. For each $v\in V$ we construct for the
data $a_v$, $b_v = (b_{v_l},\,l\in\cL(v))$, $c_v$ a Brownian motion $X^v$ on the
single vertex graph as in article~II: If $b_v=0$ then this trivially a collection of
$|\cL(v)|$ standard Brownian motions on the real line, mapped onto the external
edges of $\cG(v)$, which are such that they are absorbed at the origin
(corresponding to the vertex $v$) if $c_v=1$. If $c_v<1$ they are killed by a jump
to $\gD$ after holding the processes at the origin for an independent exponentially
distributed time of rate $a/c$. See subsection~\IIref{ssect1.4} for details. If
$b_v\ne 0$, the process $X^v$ is constructed as in section~\IIref{sect5}. The idea
(which for the case of intervals goes back to~\cite{ItMc63} and \cite{Fe54, Fe54a,
Fe57}) uses a Walsh process on $\cG(v)$ (see \cite{Wa78}, \cite{BaPi89}), and builds
in a time delay as well as killing, both on the scale of the local time at the
vertex. With appropriately chosen parameters for these two mechanisms,
theorem~\IIref{thm5.4} states that $X^v$ is a Brownian motion on $\cG(v)$ such that
its generator is the $1/2$ times the second derivative acting on $f\in
C^2_0(\cG(v))$ with boundary conditions at the vertex $v$ given by~\eqref{eq1iii}.

Next we build the graph $\cG$ by successively connecting appropriately chosen
external edges of the single vertex graphs $\cG(v)$, $v\in V$, as in
section~\ref{ssect2i}. Consider the stochastic process $X$ which is successively
constructed from the Brownian motions $X^v$ as in
subsections~\ref{ssect2ii}--\ref{ssect2iv}. Theorem~\ref{thm2xiii} states that $X$
is a Brownian motion on $\cG$ which is such that its generator has a domain which is
characterized by the same boundary conditions at each vertex $v\in V$ as for the
single vertex graphs $\cG(v)$. Therefore, $X$ is a Brownian motion on $\cG$ as in
the statement of theorem~\ref{thm1ii}, whose proof is therefore complete.

\section{Discussion of Tadpoles}    \label{sect4}

Suppose that $\cG$ is a metric graph which has one tadpole $i_t$ connected to a
vertex $v\in V$. That is, $v$ is simultaneously the initial and final vertex of $i_t$:
$\p(i_t)=(v,v)$. Figure~\ref{tad_i} shows a metric graph with a tadpole attached to
the vertex $v$.
\begin{figure}[ht]
\begin{center}
    \includegraphics[scale=.8]{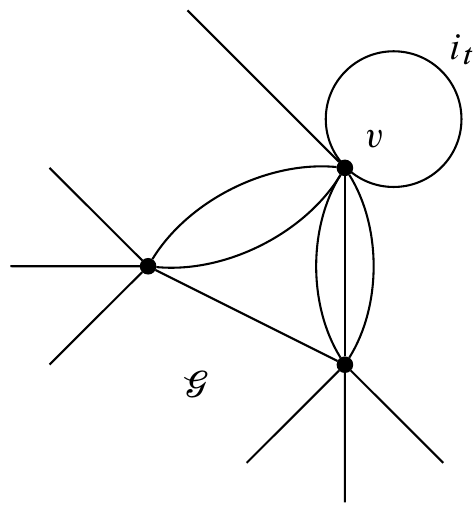}\\
    \caption{A metric graph with a tadpole $i_t$ attached to $v$.} \label{tad_i}
\end{center}
\end{figure}
Let $b_t$ be the length of $i_t$. Assume furthermore, that
we are given data $a$, $b$, $c$ as in equations~\eqref{eq1i}, \eqref{eq1ii}.
We want to construct a Brownian motion $X$ on $\cG$ the implementing boundary
conditions corresponding to these data.

Let $\cG_1$ be the metric graph obtained from $\cG$ by replacing the tadpole by two
external edges $e_1$, $e_2$, incident with $v$. Construct a Brownian motion $X_1$
on $\cG_1$ corresponding to the data $a$, $b$, $c$ as above.

Consider the real line $\R$ as a single vertex graph $\cG_2$ with the origin as the
vertex $v_0$, and edges $l_1$, $l_2$ which are isomorphic to $[0,+\infty)$,
$(-\infty, 0]$. Take a Walsh process $X_2$ on this graph which with probability
$1/2$ chooses either edge for the next excursion when at the origin. Then $X_2$ is
just a ``skew'' Brownian motion as in~\cite[p.~115]{ItMc74} which actually is not
skew. That is, it is equivalent to a standard Brownian motion on the real line.

Now join $\cG_1$ and $\cG_2$ by connecting the pairs $(e_1,l_1)$, $(e_2,l_2)$ via
two new internal edges of length $b_t/2$. Denote the resulting metric graph by $\hat
\cG$. Figure~\ref{tad_ii} shows this construction.
\begin{figure}[ht]
\begin{center}
    \includegraphics[scale=.8]{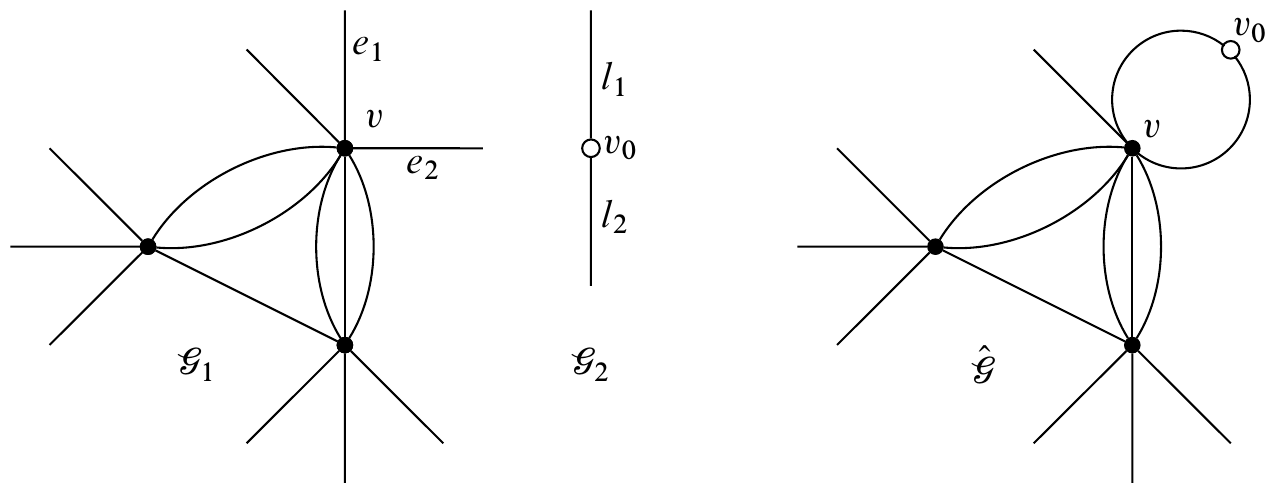}\\
    \caption{$\hat \cG$ constructed from $\cG_1$ and $\cG_2$.} \label{tad_ii}
\end{center}
\end{figure}
Construct a Brownian motion $\hat X$ on $\hat\cG$ from $X_1$ and $X_2$ as in
section~\ref{sect2}. By construction, $\hat X$ is equivalent to a standard Brownian
motion in every neighborhood of $v_0$ which is small enough such that it does not
include the vertex $v$. Therefore the additional vertex $v_0$ of $\cG$ does not
yield any non-trivial boundary condition. Thus, if we identify the open tadpole edge
$i_t^0$ with the subset of $\cG_2$ isomorphic to $(-b_t/1,b_t/2)$, then we obtain a
Brownian motion $X$ on $\cG$ implementing the desired boundary conditions.

Obviously, any (finite) number of tadpoles can be handled in the same way.

\begin{appendix}
\section{On the Crossover Times $S_n$}    \label{appA}
We recall from section~\ref{sect2} that in terms of the process $Y$ the crossover
times $S_n$, $n\in\N$, can be described as follows. Let $Y$ start in $\xi\in\cG$.
Then $S_1$ is the hitting time of $V_c\setminus\{\xi\}$ by $Y$. In particular,
$S_1>0$ for all paths of $Y$. For $n\ge 2$, $S_n$ is the hitting time after
$S_{n-1}$ of $V_c\setminus\{K_{n-1}\}$ by $Y$. Since by construction
$Y(S_{n-1})=K_{n-1}$ and the paths of $Y$ are continuous on $[0,\zeta)$, we get
$S_n>S_{n-1}$ for all paths of $Y$. Therefore
\begin{align*}
    S_n &= \inf\,\bigl\{t>S_{n-1},\,Y(t)\in V_c\setminus\{K_{n-1}\}\bigr\}\\
        &= \inf\,\bigl\{t\ge S_{n-1},\,Y(t)\in V_c\setminus\{K_{n-1}\}\bigr\}
\end{align*}
holds.

In this appendix we prove the following

\begin{lemma}   \label{lemA}
For every $n\in\N$, $S_n$ is a stopping time relative to $\cF^Y$.
\end{lemma}

\begin{proof}
Set $S_0=0$, and for $n\in\N$,
\begin{equation*}
    V_n = \begin{cases}
            V_c\setminus\{\xi\},     & \text{if $n=1$},\\[1ex]
            V_c\setminus\{K_{n-1}\}, & \text{otherwise}.
          \end{cases}
\end{equation*}
For $n\in\N$, $r\in\R_+$ define
\begin{align*}
    A_{n,r} &= \{S_n \le r < \zeta\}\\[1ex]
    B_{n,r} &= \bigcap_{m\in\N}\bigcup_{u\in\Q,\,0\le u\le r} \bigl\{S_{n-1}\le u,\,
                            d\bigl(Y(u), V_n\bigr)\le 1/m,\,r < \zeta\bigr\}.
\end{align*}
We claim that for all $n\in\N$, $r\in\R_+$,
\begin{equation}    \label{eqAi}
    A_{n,r}=B_{n,r}.
\end{equation}
To prove this claim, suppose first that $\go\in A_{n,r}$. Then $S_n(\go)$ is finite,
and therefore the set
\begin{equation*}
    \bigl\{t\ge S_{n-1}(\go),\,Y(t,\go)\in V_n(\go)\bigr\}\subset [0,\zeta(\go))
\end{equation*}
is non-empty. Thus there exists a sequence $(u_i,\,i\in\N)$ in this set which decreases
to $S_n(\go)$. Since $Y(\,\cdot\,,\go)$ is continuous
on $[0,\zeta(\go))$, it follows that $Y(S_n(\go),\go)\in V_n(\go)$. By assumption,
$S_n(\go)\le r <\zeta(\go)$, and therefore the continuity of $Y(\,\cdot\,,\go)$
on $[0,\zeta(\go))$ and $S_{n-1}(\go)<S_n(\go)$ imply that for every $m\in\N$
there exists $u\in\Q\cap[S_{n-1}(\go),r]$ so that $d(Y(u,\go),V_n(\go))\le 1/m$.
Hence $\go\in B_{n,r}$.

As for the converse, suppose now that $\go\in B_{n,r}$. Then there exists a sequence
$(u_m,\,m\in\N)$ in $\Q\cap[S_{n-1}(\go),r]$ so that $d(Y(u_m,\go), V_n(\go))$ converges
to zero as $m$ tends to infinity. Since $(u_m,\,m\in\N)$ is bounded we may assume, by
selecting a subsequence if necessary, that $(u_m,\,m\in\N)$ converges to  some
$u\in[S_{n-1}(\go),r]$ as $m\to+\infty$. Thus we find that $Y(u,\go)\in V_n(\go)$,
and therefore $Y(\,\cdot\,,\go)$ hits $V_n(\go)$ in the interval $[S_{n-1}(\go),r]$.
Consequently, $S_n(\go)\le r$, and hence $\go\in A_{n,r}$, concluding the
proof of the claim.

Next we prove by induction that for every $n\in\N$, $S_n$ is an $\cF^Y$--stopping time.
Let $n=1$. By~\eqref{eqAi} for every $r\ge 0$,
\begin{equation*}
    A_{1,r} = \bigcap_{m\in\N}\bigcup_{u\in\Q,\,0\le u\le r} \bigl\{
                            d\bigl(Y(u), V_1\bigr)\le 1/m\}\cap \{r < \zeta\bigr\}.
\end{equation*}
Clearly, $\{r<\zeta\} = \{Y(r)\in\cG\}\in\cF^Y_r$. Moreover, since $V_1$ is a
deterministic set, $d(\,\cdot\,,V_1)$ is measurable from $\cG$ to $\R_+$, and
therefore $\{d(Y(u),V_1)\le 1/m\}\in\cF^Y_u \subset \cF^Y_r$. Hence
$A_{1,r}\in\cF^Y_r$. Let $t\ge 0$, and write
\begin{align*}
    \{S_1 \le t\}
        &= \{S_1<\zeta\le t\} \cup A_{1,t}\\
        &= \Bigl(\bigcup_{r\in\Q,\,0\le r\le t} \{S_1\le r <\zeta\}\cap\{\zeta\le t\}\Bigr)\cup A_{1,t}\\
        &= \Bigl(\bigcup_{r\in\Q,\,0\le r\le t} A_{1,r}\cap\{\zeta\le t\}\Bigr)\cup A_{1,t}.
\end{align*}
Therefore $\{S_1\le t\}\in\cF^Y_t$, and hence $S_1$ is a stopping time relative to
$\cF^Y$.

Now suppose that $n\in\N$, $n\ge 2$, and that $S_{n-1}$ is an $\cF^Y$--stopping
time. We show that for all $r\in\R_+$, $A_{n,r}\in\cF^Y_r$. First we remark that
since $\cG$ is a separable metric space, the metric $d$ on $\cG$ is a measurable
mapping from $(\cG\times\cG, \cB_d\otimes\cB_d)$ to $(\R_+,\cB(\R_+))$. For example,
this follows from Theorem~I.1.10 in~\cite{Pa67}, and the continuity of
$d:\cG\times\cG \to \R_+$ when $\cG\times\cG$ is equipped with the product topology.
Consider $K_{n-1}= Y(S_{n-1})$. Since $Y$ has right continuous paths it is
progressively measurable relative to $\cF^Y$ (e.g.,
\cite[Proposition~I.4.8]{ReYo91}). Thus by Proposition~I.4.9 in~\cite{ReYo91} it
follows that $K_{n-1}$ is $\cF^Y_{S_{n-1}}$--measurable. Consequently on
$\{S_{n-1}\le u\}$, $K_{n-1}$ is $\cF^Y_u$--measurable. Equation~\eqref{eqAi} reads
\begin{equation*}
    A_{n,r} = \bigcap_{m\in\N} \bigcup_{u\in\Q,\,0\le u\le r} \{S_{n-1}\le u\}
                    \cap \bigl\{d\bigl(Y(u),V_c\setminus K_{n-1}\bigr)\le 1/m\bigr\}
                    \cap \{r< \zeta\}.
\end{equation*}
It follows that $A_{n,r}\in\cF^Y_r$, as claimed. But then $\{S_n\le t\}\in\cF^Y_t$
for all $t\in\R_+$ is proved with the same argument as at the end of the discussion
of the case $n=1$.
\end{proof}

\end{appendix}

\providecommand{\bysame}{\leavevmode\hbox to3em{\hrulefill}\thinspace}
\providecommand{\MR}{\relax\ifhmode\unskip\space\fi MR }
\providecommand{\MRhref}[2]{%
  \href{http://www.ams.org/mathscinet-getitem?mr=#1}{#2}
}
\providecommand{\href}[2]{#2}

\end{document}